\documentclass[a4paper,11pt]{amsart}
\usepackage[mac]{inputenc}
\usepackage[T1]{fontenc}
\usepackage[normalem]{ulem}
\usepackage[english]{babel}
\usepackage{a4wide}
\usepackage{amsmath, amsfonts,amssymb,amsthm}
\usepackage{enumerate}
\usepackage{bbm}
\usepackage{graphics}
\usepackage{psfrag}
\usepackage{color, epsfig}
\usepackage[only,llbracket,rrbracket]{stmaryrd}
 \usepackage[only,llbracket,rrbracket]{stmaryrd}

\def\epsilon{\varepsilon}
\def\phi{\varphi}

\def\N{\mathbb{N}}
\def\Z{\mathbb{Z}}

\def\R{\mathbb{R}} 

\def\T{\mathbb{T}}

\def\S{\mathbb{S}}

\def\fb{\mathfrak{b}}
\def\II{\parbox[][0.6cm][c]{0cm}{\ }}
\def\rot{\mathop{\mbox{curl}}}

\def\div{\mathop{\mbox{div}}}

\def\tend{\mathop{\hbox{$\longrightarrow$}}}

\theoremstyle{plain}
\newtheorem{thm}{Theorem}
\newtheorem{prop}{Proposition}[section]
\newtheorem{lem}{Lemma}[section]

\newtheorem{defi}{Definition}[section]
\theoremstyle{definition}
\newtheorem{rema}{Remark}[section]

\numberwithin{equation}{section}

\title{On the controllability of the relativistic Vlasov-Maxwell system}

\author{Olivier Glass}\address{O. Glass, CEREMADE, UMR CNRS 7534, Universit\'e Paris-Dauphine, Place du Mar\'echal de Lattre de Tassigny
75775 Paris Cedex 16, France.}
\email{glass@ceremade.dauphine.fr}
\author{Daniel Han-Kwan}\address{D. Han-Kwan, DMA, UMR CNRS 8553, \'Ecole Normale Sup\'erieure, 45 rue d'Ulm, 75230 Paris Cedex 05, France. }
\email{daniel.han-kwan@ens.fr}
 
\begin{document}

\thanks{}

\date{November 2012}
\begin{abstract}In this paper, we study the controllability of the two-dimensional relativistic Vlasov-Maxwell system in a torus, by means of an interior control. We give two types of results. With the geometric control condition on the control set, we prove the local exact controllability of the system in large time. Our proof in this case is based on the return method, on some results on the control of the Maxwell equations, and on a suitable approximation scheme to solve the non-linear Vlasov-Maxwell system on the torus with an absorption procedure.

Without geometric control condition, but assuming that a strip of the torus is contained in the control set and under certain additional conditions on the initial data, we establish a controllability result on the distribution function only, also in large time. Here, we need some additional arguments based on the asymptotics of the Vlasov-Maxwell system with large speed of light and on our previous results
concerning the controllability of the Vlasov-Poisson system with an external magnetic field \cite{GHK}.
\end{abstract}

\maketitle

 
\section{Introduction}

\subsection{The relativistic Vlasov-Maxwell system}

Consider the relativistic Vlasov-Maxwell system:
\begin{equation} \label{VlasovMaxwell}
 \left\{ \begin{aligned}
& \partial_t f + \hat{v}\cdot\nabla_x f +  \left(E   + \frac{1}{c}\hat{v} \wedge B \right) \cdot \nabla_v f=\mathbbm{1}_\omega G, \quad \quad t\geq0, \, x \in \T^3:= \R^3/\Z^3, \, v \in \R^3,\\
& \partial_t E + c \operatorname{curl} B = -\int_{\R^3} f \hat v \, dv,  \quad \partial_t B + c \operatorname{curl} E=0, \\
& \operatorname{div} E  = \int_{\R^3} f \, dv- \int_{\R^3\times \T^3} f \, dv \,dx, 
 \quad \operatorname{div} B = 0, \\
& f_{| t=0} = f_0, \quad E_{| t=0} = E_0, \quad B_{| t=0} = B_0,
\end{aligned} \right.
\end{equation}
 where  $c>0$ is the \emph{speed of light} and
$$\hat v :=\frac{v}{\sqrt{1+\frac{|v|^2}{c^2}}}$$
is the relativistic velocity. 
The so-called distribution function $f(t,x,v)$ describes the statistical distribution of a population of electrons in a collisionless plasma: the quantity $f(t,x,v) \, dx \, dv$ can be interpreted as the number of particles at time $t$ whose position is close to $x$ and velocity close to $v$. The term $\mathbbm{1}_\omega G$ is a source in the Vlasov equation.
As usual, $E$ and $B$ stand for the electric and magnetic fields, which are solutions of the Maxwell equations. In the equation on $\operatorname{div} E$, the term $\int_{\R^3\times \T^3} f \, dv \,dx$ stands for a neutralizing background of fixed ions. We take as a convention here that the Lebesgue measure of the torus is equal to $1$.
  
In this paper, we will focus on the two-dimensional version of the Vlasov-Maxwell system, which means that $(x,v) \in \mathbb{T}^2 \times \mathbb{R}^2$. This corresponds to the physical situation in the usual three-dimensional geometry where:
\begin{equation}
  \begin{aligned}
  &E=(E_1(t,x),E_2(t,x),0), \\
   &B= (0,0, b(t,x)).
 \end{aligned}
  \end{equation}
In this framework, noticing that the divergence-free condition is automatically satisfied for $B$, the Vlasov-Maxwell system can be rewritten in the simplified form:
\begin{equation}
\label{VlaM}
    \left\{
  \begin{aligned}
  & \partial_t f + \hat{v}\cdot\nabla_x f + \left(E   + \frac{1}{c}\hat{v}^\perp b \right) \cdot \nabla_v f=\mathbbm{1}_\omega G, \quad \quad t\geq0, \, x \in \T^2, \, v \in \R^2,\\
& \partial_t E_1 - c \partial_{x_2} b=-\int_{\R^2} f \hat v_1 \, dv, \quad \partial_t E_2 + c \partial_{x_1} b=-\int_{\R^2} f \hat v_2 \, dv,  \\
& \partial_t b + c\left[  \partial_{x_1} E_2 -  \partial_{x_2} E_1\right]=0, \\
& \partial_{x_1} E_1 +  \partial_{x_2} E_2 =  \int_{\R^2} f \, dv- \int_{\R^2\times \T^2} f \, dv \, dx .
 \end{aligned}
 \right.
  \end{equation}
 If $\hat{v}= (\hat{v}_1,\hat{v}_2)$, we set $\hat{v}^\perp := (\hat{v}_2,-\hat{v}_1)$.
For simplicity, we shall also denote in the sequel $\operatorname{curl} b:=\nabla^\perp b = (\partial_{x_2} b, -\partial_{x_1} b)$.  

 We are interested in the \emph{controllability} properties of this system, by means of an interior control $G$, localized in the space variable in an open subset $\omega$ of the torus, but without localization in the velocity variable. The basic controllability question is the following one. Given $(f_0, E_0, B_0)$, $(f_1, E_1, B_1)$ belonging to a suitable function space and given some control time $T>0$, is it possible to steer the system from the first state to the second one in time $T$ with a well chosen control function $G$ ? In other words, we ask the solution of \eqref{VlaM} to satisfy:
\begin{equation}
 f_{| t=T} = f_1, \quad E_{| t=T} = E_1, \quad B_{| t=T} = B_1.
 \end{equation}
 Note that in order to preserve electroneutrality, the control $G$ has to satisfy
 $$
 \int_{\T^2 \times \R^2} \mathbbm{1}_\omega(x) G(t,x,v) \, dx \, dv =0, \quad \forall t \in [0,T].
 $$

\subsection{Geometric assumptions and main results of the paper}

We prove two kinds of controllability results: for the first one we assume that the celebrated geometric control condition is satisfied for $\omega$. For the second one, we assume another type of condition on $\omega$, as well as other conditions on the initial magnetic field and/or the speed of light.

\subsubsection{With the geometric control condition}
The Vlasov-Maxwell can be seen as a hyperbolic system, made of a transport equation (the Vlasov part) and a coupling between two wave equations (the Maxwell part).
Concerning the problem of controllability of the linear wave equation, it is well-known that the geometric properties of the control set $\omega$ play a key role.
 In the seminal works \cite{RT} and  \cite{BLR2,BLR}, J. Rauch and M. Taylor and C. Bardos, G. Lebeau and J. Rauch introduced the so-called \emph{geometric control condition} on $\omega$, which roughly means that every ray of geometric optics meets $\omega$ before some fixed time, and proved estimates which imply the exact controllability of the wave equation.
 
 In our case, the geometric control condition can be formulated as follows:
\begin{defi}
The open subset $\omega$ of $\T^2$ satisfies the geometric control condition (GCC) if:
\begin{multline} \label{GeometricConditionOmega}
\text{For any } x \in \mathbb{T}^2 \text{ and any direction } e \in \mathbb{S}^{1}, \text{ there exists } y \in \R^{+} \text{ such that } x+ye \in \omega.
\end{multline}
\end{defi}
Conversely, it is well known (see N. Burq and P. Gérard \cite{BG}) that the geometric control condition is not only sufficient but is also actually a \emph{necessary} condition to get exact controllability of the wave equation.
 
It is hence natural to start by assuming that $\omega$ does satisfy the geometric control condition.
Our first result asserts that with this assumption, local controllability near $0$ of the Vlasov-Maxwell system holds  when the controllability time is large enough. This restriction is expected because of the finite propagation speed both for the electromagnetic waves and for the charged particles.

\begin{thm}[With geometric control condition] \label{TheoGCC}
Assume that $\omega$ satisfies the geometric control condition.
There exists $T_0>0$, such that for any $T>T_0$, there exists  $\kappa>0$ satisfying the following. If for $i=0,1$,  $(f_i,E_i,B_i) \in H^3(\T^2\times \R^2) \times H^3(\T^2) \times  H^3(\T^2)$ are such that $f_i$ is compactly supported in $v$ and satisfy the compatibility conditions
\begin{align}
& \operatorname{div} E_i = \int_{\R^2} f_i \, dv - \int_{\T^2\times\R^2} f_i \, dv \, dx, \\
& \int_{\T^2\times\R^2} f_0 \, dv \, dx =  \int_{\T^2\times\R^2} f_1 \, dv \, dx,
\end{align}
as well as the smallness assumption
\begin{align}
\|  (f_i,E_i,B_i)\|_{H^3} \leq \kappa,  \quad i=0,1,
\end{align} 
then there exists a control $G \in H^2([0,T] \times \T^2 \times \R^2)$, supported in $\omega$, such that the corresponding solution $(f,E,B)$ to the Vlasov-Maxwell system \eqref{VlaM} with initial data $(f_0,E_0,B_0)$ satisfies:
\begin{equation*}
 f_{| t=T}=f_1, \quad E_{| t=T} = E_1, \quad B_{| t=T} = B_1.
\end{equation*}
\end{thm}
 
 
\begin{rema}
The solution that we build involves an electric field $E$ and a magnetic field $B$ which belong to the Sobolev space $C([0,T]; H^3(\T^2))$, which is imbedded in $C([0,T]; C^1(\T^2))$, so that characteristics are well defined in the classical sense.
\end{rema}
 
  
\begin{rema}
It is possible to generalize this theorem in two directions: relax the smallness assumption on $B_i$ and replace the assumption of compact support in $v$ by an assumption of exponential decrease in $v$. We refer to Subsections \ref{gene1} and \ref{gene2}.
\end{rema}

 \subsubsection{With the strip assumption}
Since GCC is necessary and sufficient for the controllability of the wave equation, it seems very involved (although maybe not impossible)  to get the result of Theorem \ref{TheoGCC} without assuming GCC.
 
 Nevertheless, the situation is more favorable  if we focus on the system of particles only and ``forget'' about the electromagnetic field: this corresponds to the control of the distribution function $f$ only. Indeed, for the kinetic transport equation, we do have more flexibility than for the wave equation and it is natural to try to relax as much as possible the assumptions on $\omega$.
 In this work, we are able to prove some results in this direction. Namely, we prove a {controllability} result on the distribution function only with a geometric condition which does not involve GCC (and which is not implied by GCC).

  \medskip

This geometric assumption on the control set will be referred to as the \emph{strip assumption} (loosely speaking, as the name suggests, this corresponds to the case where $\omega$ contains a strip).

\begin{defi}
We say that $\omega$ satisfies the strip assumption if there exists a straight line of $\R^{2}$ whose image ${\mathcal H}$ by the canonical surjection $s:\R^{2} \rightarrow \T^{2}$ is closed and included in $\omega$. We call $n_{H}$ a unit vector, orthogonal to ${\mathcal H}$. For $l>0$, we denote 
\begin{equation*}
{\mathcal H}_{l}:= {\mathcal H} + [-l,l] n_{H}.
\end{equation*}
\end{defi}

Since ${\mathcal H}$ is closed in $\T^{2}$, note that we can define ${d}>0$ such that 
\begin{equation}
\label{H2d}
{\mathcal H}_{2d} \subset \omega,
\end{equation}
and such that $4{d}$ is less than the distance between two successive lines in $s^{-1}({\mathcal H})$.

We first prove an ``asymptotic'' controllability result, for which the speed of light is considered as a parameter. This next result will be valid for a large speed of light. 


%
%
 
 
\begin{thm}[With the strip assumption]
\label{Theo2}
Assume that $\omega$ satisfies the strip assumption.  Let  $b_0 \in H^3(\T^2;\R)$ be a magnetic field such that
\begin{equation*}
\int_{\T^2} b_0 \, dx \neq 0.
\end{equation*}
There exists $T_0>0$ such that for $T \geq T_0$, there exist $\kappa>0$ and $c_T>0$ satisfying the following.
If $c \geq c_T$, if  $(f_0,E_0) \in H^3(\T^2\times \R^2) \times H^3(\T^2) $  and $f_1 \in H^3(\T^2\times \R^2)$ are  such that $f_0, f_1$ are compactly supported in $v$ and satisfy
\begin{align}
\label{Th2C1}
& \operatorname{div} E_0 =\int_{\R^2} f_0 \, dv - \int_{\T^2\times\R^2} f_0 \, dv \, dx, \\  
\label{Th2C2}
& \int_{\T^2\times\R^2} f_0 \, dv \, dx =  \int_{\T^2\times\R^2} f_1 \, dv \, dx, \\ 
\label{Th2C3}
& \| (f_0,E_0)\|_{H^3} \leq \kappa, \quad \| f_1\|_{H^3} \leq \kappa, 
\end{align} 
and if
\begin{align}
\label{Th2C4}
& B_0 = c \, b_0 , 
\end{align} 
then there exists a control $G \in H^2([0,T] \times \T^2 \times \R^2)$, supported in $\omega$, such that the corresponding solution $(f,E,B)$ to the Vlasov-Maxwell system \eqref{VlaM} with initial data $(f_0,E_0,B_0)$ satisfies:
\begin{equation*}
 f_{| t=T}=f_1.
\end{equation*}
\end{thm}
\begin{rema}
The scaling in \eqref{Th2C4} is in fact natural if we want to recover the convergence of Vlasov-Maxwell to Vlasov-Poisson with an external magnetic field, in the classical limit $c \rightarrow + \infty$: see Paragraph \ref{VMVP}. \par
From the physical point of view, one can remark that $B_0$ is in CGS units, namely Gauss (which is a theoretical physics unit), whereas $\frac{B_0}{c}$ is in SI units, namely Teslas (``usual'' unit). Thus, imposing that $B_0=  c \, b_0$ is equivalent to imposing that the intensity keeps a non-zero value when $c$ becomes large. So this assumption is also natural from this point of view.
\end{rema}
As a consequence of the previous result, we are also be able to prove the following theorem, in which the speed of light is considered as a fixed parameter, let us say $c=1$.
\begin{thm}[With the strip assumption]
\label{Theo}
Let $c=1$. Assume that $\omega$ satisfies the strip assumption.
 Let  $b_0 \in H^3(\T^2)$ be a magnetic field such that
\begin{equation*}
\int_{\T^2} b_0 \, dx \neq 0.
\end{equation*}
Then there exist $T_1>0$ and $\lambda_{1}>0$ such that the following holds.
For any $\lambda \in (0,\lambda_{1}]$, for any $T\geq T_1/ \lambda$, there exists $\kappa>0$ such that, if  $(f_0,E_0) \in H^3(\T^2\times \R^2) \times H^3(\T^2)$ and $f_1 \in H^3(\T^2\times \R^2)$ are compactly supported in $v$ and satisfy
\begin{align}
\label{Th3C1} &   \operatorname{div} E_0 =\int_{\R^2} f_0 \, dv - \int_{\T^2\times\R^2} f_0 \, dv \, dx, \\ 
\label{Th3C2} & \int_{\T^2\times\R^2} f_0 \, dv \, dx =  \int_{\T^2\times\R^2} f_1 \, dv \, dx, \\
\label{Th3C3}
& \| (f_0,E_0)\|_{H^3} \leq \kappa, \, \| f_1\|_{H^3} \leq \kappa, 
\end{align}
and if
\begin{align}
\label{Th3C4}
&B_0(x)= \lambda \, b_0(x),
\end{align} 
then there exists a control $G \in H^2([0,T] \times \T^2 \times \R^2)$, supported in $\omega$, such that the corresponding solution $(f,E,B)$ to the Vlasov-Maxwell system \eqref{VlaM} with initial data $(f_0,E_0,B_0)$ satisfies:
\begin{equation}
f_{|t=T}=f_1.
\end{equation}
\end{thm}
\begin{rema}
The choice $c=1$ in Theorem \ref{Theo} is naturally arbitrary (we could have taken any positive constant).
Note also that taking $\lambda$ small corresponds to an assumption of smallness on the initial magnetic field $B_0$. But this yields a result of controllability for a large time. As we will see in the proof, the effect of the magnetic field is indeed crucial to get the result.
\end{rema}

  \subsection{A short review on the Cauchy problem} Let us now provide a short review on the classical Cauchy theory for the Vlasov-Maxwell equation.
 To our knowledge, the very first results concerning these equations concern the local in time existence of classical solutions,  and were obtained in independent works by K. Asano \cite{ASA}, P. Degond \cite{DEG} and  S. Wollmann \cite{Wol1,Wol2}. All these contributions are based on Sobolev type function spaces. For a ``sharper'' theory based on $C^1$ spaces, we refer to the fundamental works of R. Glassey and W. Strauss \cite{GS86,GS87,GS89}, where the structure of the Vlasov-Maxwell equation (kinetic transport + wave) plays a crucial role. Concerning the global theory, there have also been several remarkable contributions: in 3D, for equations set on the whole space, R. Glassey and W. Strauss gave in the pionneering work \cite{GS86} a famous criterion for global existence (see also the more recent contributions of S. Klainerman and G. Staffilani \cite{KS}, and F. Bouchut, F. Golse and C. Pallard \cite{BGP} for alternative proofs); further studies around this criterion have been at the core of many papers (see for instance \cite{GS-CMP,SCH-ind,PAL}), but the unconditional global existence in three dimensions is still an open problem. In 2D, in the whole space case, the problem of global existence was settled by R. Glassey and J. Schaeffer \cite{GSC98} as well as for the so-called 2D-1/2 case \cite{GSC97}.
 
 We also mention the fundamental contribution of R.J. DiPerna and P.-L. Lions, who built in \cite{DPL} global weak solutions (in the energy space) to the Vlasov-Maxwell system, but we shall not use this type of solutions and will only consider classical solutions, for which characteristics are well-defined.
%
%
%
%
\subsection{Organization of the following of the paper}
Let us describe how the paper is organized.
In Section \ref{sec1}, we explain the strategy of the proofs. In Sections \ref{sec2} and \ref{sec3}, we construct some relevant reference solutions, in order to apply the return method. Section \ref{sec2} considers the case of Theorem \ref{TheoGCC}, which corresponds to the situation with geometric control condition, while Section \ref{sec3} deals with the case of Theorems \ref{Theo2} and \ref{Theo} (with the strip assumption). For the latter, we will revisit the case of Vlasov-Poisson with external force fields of \cite{GHK} and rely on the fact that the dynamics of Vlasov-Poisson with external magnetic field is a good approximation of the dynamics of Vlasov-Maxwell, when the speed of light is large. Finally, in Section \ref{sec4}, we present the approximation scheme which will allow to give a solution of the full non-linear problems and conclude the proofs of Theorems \ref{TheoGCC}, \ref{Theo2} and \ref{Theo}. This scheme involves an absorption procedure which yields some new technical difficulties (that do not appear when one considers the Cauchy problem for the relativistic Vlasov-Maxwell system without source). Finally we give in an Appendix the proof of a key lemma for Theorems \ref{Theo2} and \ref{Theo}.
\section{Strategy of the proofs}
\label{sec1}
As a preliminary, we set up in this section the strategies that will be implemented in this paper.
\subsection{Some notations} Let us first give some useful notations.

For $T>0$, we denote $Q_{T}:=[0,T] \times \T^{2} \times \R^{2}$. For a domain $\Omega$, we write $C_{b}^{l}(\Omega)$, for $l \in \N$, for the set $C^{l}(\Omega) \cap W^{l,\infty}(\Omega)$. 
For any $l \in \N$, we denote by $H^l(\T^2 \times \R^2)$ the usual Sobolev space defined in $\T^2 \times \R^2$. \par
For $x$ in $\T^{2}$ and $r >0$, we denote by $B(x,r)$ the open ball
of center $x$ and radius $r$, and by $S(x,r)$ the corresponding
circle. All radii will always be chosen small enough so that
$S(x,r)$ does not intersect itself. 
\par
Let us also introduce some notations regarding characteristics and auto-induced electromagnetic field. Let $F(t,x,v)$ some force field with Lispchitz regularity and sublinear at infinity in $v$. Let $s\geq 0$, which corresponds to an ``initial'' time and $(x,v) \in \T^2 \times \R^2$.

We call $(X(t,s,x,v),V(t,s,x,v))$ the \emph{relativistic} characteristics associated with $F$, the solutions of the system of ODEs:
\begin{equation*}
\left\{ \begin{aligned}
& \frac{dX}{dt}=\hat{V},\\
& \frac{dV}{dt}= F(t,X,V), \\
& X(s,s,x,v)= x, \quad V(s,s,x,v)= v,
\end{aligned} \right.
\end{equation*}
with $\hat{V}= \frac{V}{\sqrt{1+ \frac{|V|^2}{c^2}}}$. 

We can also define  $(\mathcal{X}(t,s,x,v),\mathcal{V}(t,s,x,v))$ the \emph{classical} characteristics associated with $F$:
\begin{equation*}
\left\{ \begin{aligned}
& \frac{d\mathcal{X}}{dt}={\mathcal{V}},\\
& \frac{d\mathcal{V}}{dt}= F(t,\mathcal{X},\mathcal{V}), \\
& \mathcal{X}(s,s,x,v)= x, \quad \mathcal{V}(s,s,x,v)= v.
\end{aligned} \right.
\end{equation*}

In the following, when not otherwise specified, the characteristics under study are relativistic.

Note that the characteristics (relativistic or classical) are well defined by the classical Cauchy-Lipschitz theorem. Usually, when there is no ambiguity, we will simply write $(X,V)$ instead of the full notation $(X(t,0,x,v),V(t,0,x,v))$. \par
Finally, for any distribution function $f(t,x,v)$ sufficiently smooth and with sufficient decay in velocity, we will denote $E^f(t,x)$ and $B^f(t,x)$ the auto-induced force fields, which are solutions of the Maxwell equations:
\begin{equation*}
\left\{ \begin{aligned}
& \partial_t E^f + c \operatorname{curl} B^f = -\int_{\R^2} f \hat v \, dv,  \quad \partial_t B^f + c \operatorname{curl} E^f=0, \\
& \operatorname{div} E^f  = \int_{\R^2} f \, dv- \int_{\R^2\times \T^2} f \, dv \, dx, \quad \operatorname{div} B^f = 0, \\
\end{aligned} \right. 
\end{equation*}
with some relevant initial data (depending on the context).

\subsection{The return method}
Since we intend to show some controllability properties on the Vlasov-Maxwell system, for small data (that are close to $0$), it would be natural to apply the following proof scheme:
\begin{enumerate}
\item Linearize the equations around $0$ and prove some controllability properties on these linear equations.
\item Prove that these properties are somehow preserved by the full equations, by using some fixed point theorem.
\end{enumerate}
Unfortunately, we are in a situation where this scheme fails. 
If we  consider the linearized Vlasov-Maxwell equation,  around the trivial state $(\overline{f}, \overline{E}, \overline{B})=(0,0,0)$, we obtain the following equation:
%
\begin{equation} \label{freee}
\partial_t f + \hat{v}\cdot\nabla_x f = \mathbbm{1}_\omega G,
\end{equation}
which is the \emph{relativistic free transport equation}. The characteristics associated with the transport part are simply straight lines. By Duhamel's formula, we obtain the explicit representation for $f$:
\begin{equation*}
f(t,x,v)=f_0(x-t\hat{v},v) + \int_0^t  (\mathbbm{1}_\omega G)(s,x-(t-s)\hat{v},v)ds.
\end{equation*}
One can observe that there are two types of obstruction to controllability for the (linear) relativistic free transport equation.
\begin{itemize}
\item (Small velocities) The first obstruction concerns the small velocities. If the velocity $\hat{v}$ of a charged particle driven by \eqref{freee} is not large enough, then it will not be able to reach the control zone $\omega$ in the desired time $T$, so that $(\mathbbm{1}_\omega G)(s,x-t\hat{v},v)=0$ for all $t \in [0,T]$. In particular, if $c$ is not large enough with respect to the desired control time, then this obstruction appears for many trajectories.
\item (Bad directions) The second obstruction is of geometric  type: it is possible that a particle has initially a bad direction, so that, even if the modulus of its velocity and the speed of light are very large,  it will never reach the control zone. Thus we cannot influence the value of the distribution function along this trajectory. There is no such problem when $\omega$ satisfies the geometric control condition, but this problem can occur under the strip assumption.
\end{itemize}
\begin{figure}[!ht]
\begin{center}
	\resizebox{!}{4cm}{\input{ObstructionsVlasov.pstex_t}}
\end{center}
\begin{caption}
{The two obstructions for controllability}
\end{caption}
\end{figure}
%
%
%
To overcome this difficulty, we shall rely on  the return method of J.-M. Coron (see the book  \cite{COR} for many examples of applications of this method for several non-linear PDEs, especially coming from fluid mechanics). The idea is 
to  build some relevant \emph{reference} solution of the full PDEs (with a well-chosen source supported in the control set $\omega$), starting from the trivial state $0$ and reaching again $0$ after some time fixed in advance, around which the linearized equations do enjoy nice controllability properties (as opposed to the trivial solution $0$). Then, in the fixed point scheme designed to construct a solution of the non-linear problem, we will linearize the PDEs around this reference solution, and not the trivial one.
 
For our problem, the key argument is to find a reference solution $(\overline{f},\overline{E},\overline{B})$ to the relativistic Vlasov-Maxwell system, with some suitable source in $\omega$, starting from $(0,0,0)$ and coming back to $(0,0,0)$, and such that the characteristics associated with $\overline{E} + \frac{\hat{v}^\perp}{c} \overline{B}$ satisfy:
\begin{equation} \label{returnn}
\forall x \in \T^2, \forall v \in \R^2, \exists t \in [0,T], \quad X(t,x,v) \in \omega.
\end{equation}
Note indeed that the linearized equation reads:
\begin{equation*}
\partial_t f + \hat{v}\cdot\nabla_x f + (\overline{E} + \frac{\hat{v}^\perp}{c} \overline{B})\cdot\nabla_vf  = \mathbbm{1}_\omega G,
\end{equation*}
so that \eqref{returnn} means that the two above obstructions do not occur.

It is clear that this condition on the characteristics is necessary to get some controllability properties on the linearized equations; we will show later that this will also be sufficient to get local controllability for the full equations.

Actually, in view of the fixed point procedure which we use to build a solution of the non-linear problem, we will need to prove a slightly refined version of \eqref{returnn}:
\begin{enumerate}
\item In the case with GCC, we consider open sets $\omega', \omega''$ contained in $\omega$, which still satisfy GCC, such that 
$$\omega' \subset \overline{\omega'} \subset \omega'' \subset \overline{\omega''} \subset \omega.$$
We ask that the characteristics associated with $\overline{E} + \frac{\hat{v}^\perp}{c} \overline{B}$ satisfy:
\begin{equation} \label{return}
\forall x \in \T^2, \forall v \in \R^2, \exists t \in [0,T], \quad X(t,x,v) \in \omega'.
\end{equation}
\item In the case with the strip assumption, we consider $x_0 \in \omega$ and  $r_0>0$ such that $B(x_0,2r_0)$ is contained in $\omega$. We ask that the characteristics associated with $\overline{E} + \frac{\hat{v}^\perp}{c} \overline{B}$ satisfy:
\begin{equation} \label{return2}
\forall x \in \T^2, \forall v \in \R^2, \exists t \in [0,T], \quad X(t,x,v) \in B(x_0,r_0).
\end{equation}
\end{enumerate}
%
%
%
%
%
%
%
\subsection{The case of Theorem \ref{TheoGCC}}
With geometric control condition on the control set $\omega$, we can rely on \emph{exact controllability} properties for the Maxwell equations, which were already studied by K.-D. Phung in \cite{PHU} (actually in the more involved context of open sets of $\R^3$ with boundary). This controllability tool will be given in Theorem \ref{controlMaxwell}. The crucial point is that there exist \emph{stationary} solutions of Maxwell which are relevant for \eqref{return}. This will allow us to build a suitable reference solution as follows. Starting from zero, by using a suitable control $G$, one can emit a wave from $\omega$, so that the electromagnetic field reaches such a stationary solution. Then one stays at this point for some time interval, and then returns to zero by emiting again a wave from $\omega$.
\subsection{The case of Theorem \ref{Theo2}} 
\label{VMVP}
Without the geometric control condition, it would be at first sight tempting to try to use some results based on the \emph{approximate} controllability of wave equations (such results are for instance given for the wave equations in the survey of D. Russell \cite{RUS} and in the book of J.-L. Lions \cite{LIO}). Unfortunately, there are at least two obstructions. First, it is not clear at all that such results hold for the Maxwell equations (because of some consistency issues, see equation \eqref{consmass} below, that do not appear for the wave equations). Furthermore, such a strategy is likely to fail, since one has to find a way to ``sustain'' for some time the appropriate electromagnetic field $(E,B)$, so that the trajectories of the particles get sufficiently affected by them. This was possible with the geometric condition condition on $\omega$, only because it is possible to \emph{exactly} reach some \emph{stationary} solution of the Maxwell equations. With only approximate controllability, it seems no longer possible to do so.

Instead, we propose a strategy based on the \emph{large speed of light asymptotics} (in other words the non-relativistic limit) of the Vlasov-Maxwell equation.
From the physical point of view, this limit is well understood: it is well-known that as the speed of light goes to infinity, the Maxwell equations are approximated by the Poisson equation: this corresponds to the so-called electrostatic approximation, which is valid when the speed of light is large enough compared to the typical velocities of the system. Concerning the transport part, that is, the relativistic Vlasov equation, one can readily see that the effect of a large speed of light is to lead, at least formally, to the ``classical'' transport setting, that is:
$$
\hat{v}= \frac{v}{\sqrt{1+\frac{|v|^2}{c^2}}} \,  \tend_{c \rightarrow + \infty} \, v.
$$
When $c \rightarrow + \infty$, one can indeed rigorously prove the convergence of classical solutions of the Vlasov-Maxwell system towards classical solutions of the Vlasov-Poisson system. This was done independently in works by K. Asano and S. Ukai \cite{AU}, P. Degong \cite{DEG} and J. Schaeffer \cite{SCH}, in the framework of equations set in the whole space $\R^3$. We recall that the Vlasov-Poisson (without magnetic field) system reads:
\begin{equation*}
\left\{ \begin{aligned}
& \partial_t f + {v}\cdot\nabla_x f +  \operatorname{div}_v \left[\nabla_x \varphi f\right]=0,\\
& \Delta_x \varphi = \int f \, dv  - \int f \, dv \, dx, \\
& f_{| t=0} = f_0.
\end{aligned} \right.
\end{equation*}
%

Somehow motivated by the convergence of Vlasov-Maxwell to Vlasov-Poisson, we intend to show that the study  of the controllability of the Vlasov-Poisson system with \emph{external force fields}, which was done in our previous work \cite{GHK}, can be useful in the Maxwell context. This means that we rely on the fact that the Vlasov-Maxwell dynamics is close to that of the simpler Vlasov-Poisson with external magnetic field system, when $c$ is very large.
The crucial points are Lemma \ref{PropSolRefCM} (which shows how magnetic fields have a ``bending'' effect on characteristics and help to overcome the bad directions) and the Approximation Lemma \ref{approx} which allows us in some sense to study a Poisson equation instead of Maxwell equations. To apply the approximation lemma, some difficulties appear, due to some consistency issues in the Maxwell and Poisson equations, and we will rely on the strip assumption on the control set $\omega$ to solve them.

\begin{rema}
Taking $B_0 = c \, b_0 $ in \eqref{Th2C4}, where $c$ is the speed of light, is the ``right'' scaling  if one wants to recover the Vlasov-Poisson system with an external magnetic field $b_0$  in the classical limit $c\rightarrow + \infty$ (taking $B_0 =b_0$ yields no magnetic field in \cite{SCH}.) 
\end{rema}

\begin{rema}
 In some sense, the spirit of our proof is quite similar to that of the proof of approximate controllability of the incompressible Navier-Stokes equations with Navier slip boundary conditions, by J.-M. Coron \cite{COR2}, whose basic principle relies on the fact that for small viscosities, solutions of the incompressible Navier-Stokes system behave like solutions of the incompressible Euler equations.
 \end{rema}

\subsection{The case of Theorem \ref{Theo}: scaling properties of the Vlasov-Maxwell system}
Finally, for the proof of Theorem \ref{Theo}, we will also exploit some scaling properties of Vlasov-Maxwell, that we recall here.
One can observe that the Vlasov-Maxwell system is ``invariant'' by some change of scales.
This is a kind of generic property to Vlasov equations.
More precisely, when $(f,E,B)$ is a solution
of Vlasov-Maxwell in $[0,T] \times \mathbb{T}^{2} \times
\mathbb{R}^{2}$, then for $\lambda \not = 0$, 
\begin{equation} \label{lamb}
f^{\lambda}(t,x,v) :=  f(\lambda t,x,v/\lambda),
\end{equation}
\begin{equation} \label{ChgtEchelleE}
E^{\lambda}(t,x) := \lambda^{2} E(\lambda t,x),
\end{equation}
%
\begin{equation} \label{ChgtEchelleB}
B^{\lambda}(t,x) := \lambda^{2} B(\lambda t,x).
\end{equation}
is still a solution in $[0,T/\lambda] \times \mathbb{T}^{2} \times
\mathbb{R}^{2}$ of a rescaled Vlasov-Maxwell system:%
\begin{equation} \label{rescale}
\left\{ \begin{aligned}
& \partial_t f^{\lambda} + \hat{v} \cdot \nabla_x f^{\lambda} 
  + \operatorname{div}_v \left[(E^{\lambda}+ \frac{1}{c\lambda}\hat{v} \wedge B^{\lambda} )f^{\lambda}\right]=\mathbbm{1}_\omega G \\
& \partial_t E^{\lambda} + {c}{\lambda} \operatorname{curl} B^{\lambda} = -\int_{\R^2} f^{\lambda}\hat v \, dv,  
  \quad \partial_t B^{\lambda} + {c}{\lambda} \operatorname{curl} E^{\lambda}=0, \\
& \operatorname{div} E^{\lambda}  = \int_{\R^2} f^{\lambda} \, dv- \int_{\R^2\times \T^2} f^{\lambda} \, dv \, dx, 
  \quad \operatorname{div} B^{\lambda} = 0, \\
& f^{\lambda}_{| t=0} = f^{\lambda}_0, \quad E^{\lambda}_{| t=0} = E^{\lambda}_0, \quad B^{\lambda}_{| t=0} = B^{\lambda}_0,
\end{aligned} \right.
\end{equation}
with $\hat v =\frac{v}{\sqrt{1+\frac{|v|^2}{(c\lambda)^2}}}$.
%

%
This implies that choosing a large parameter $\lambda \gg 1$ \emph{artificially} increases the speed of light. This is the key point to deduce Theorem \ref{Theo} from Theorem \ref{Theo2}.
Another important choice is $\lambda= -1$, which shows that the equations are \emph{reversible}. This will allow to simplify the proof of Theorem \ref{TheoGCC}: we will only prove the result for $f_1 \equiv 0$ outside $\omega$, $(E_1,B_1)=(0,0)$ and deduce the general case by reversibility (actually it is possible to avoid this argument, see Subsection \ref{5.5}).
 
In another direction, we mention that the choice $0< \lambda \ll 1$ can be useful to prove \emph{global} controllability results for Vlasov systems (that correspond to results without restriction on the size of the data), and was indeed a crucial argument in the Vlasov-Poisson case in \cite{GLA,GHK}. Unfortunately, we will not be able to do so for this problem, since it is not compatible with our method using the bending effect of the magnetic field. 
%
%
%
%
%
%
%
%
%
%
%
%
%
%
%
%
%
\section{The reference solution with geometric control condition}
\label{sec2}
In this section, we assume that the geometric control condition is satisfied by the control set $\omega$, and we build the reference solution $(\overline{f},\overline{E},\overline{B})$ which is central in the proof of Theorem \ref{TheoGCC}. To this purpose, we first describe in Subsection \ref{SSmax} some results on the controllability of the Maxwell equations. These results are needed for the construction of the reference solution, which is given in Subsection \ref{SScons} (in particular we will explain how to obtain a relevant distribution function $\overline{f}$).
\subsection{Controllability of the Maxwell equations with the geometric control condition}
\label{SSmax}

The controllability problem which we consider in this paragraph is the following one. Let $T>0$ be some control time and let $(E_0, B_0,E_1,B_1)$ be some electromagnetic field (belonging to some appropriate function set), the question is to find $\tilde j$ and $\tilde \rho$ supported in $\omega$ such that 
\begin{equation*}
\left\{ \begin{split}
& \partial_t \tilde B + c \operatorname{curl} \tilde E = 0, \quad \partial_t \tilde E + c \operatorname{curl} \tilde B=-\tilde j, \\
& \operatorname{div} \tilde E  = \tilde \rho - \int_{\T^2}  \tilde \rho \, dx,  \quad \operatorname{div} \tilde B = 0, \\
& \tilde E_{| t=0} = E_0, \quad \tilde B_{| t=0} = B_0, \\
& \tilde E_{| t=T} = E_1, \quad \tilde B_{| t=T} = B_1. 
\end{split} \right.
\end{equation*}
Using the reversibility of the Maxwell equations, it is straightforward to see that this is equivalent to the next controllability problem. Let $(E_0, B_0)$ be some electromagnetic field, find $\tilde j$ and $\tilde \rho$ supported in $\omega$ such that 
\begin{equation*}
\left\{ \begin{split}
& \partial_t \tilde B + c \operatorname{curl} \tilde E = 0, \quad \partial_t \tilde E + c \operatorname{curl} \tilde B=-\tilde j, \\
& \operatorname{div} \tilde E  = \tilde \rho - \int_{\T^2} \tilde \rho \, dx,  \quad \operatorname{div} \tilde B = 0, \\
& \tilde E_{| t=0} = E_0, \quad \tilde B_{| t=0} = B_0, \\
& \tilde E_{| t=T} = 0, \quad \tilde B_{| t=T} = 0. 
\end{split} \right.
\end{equation*}
This controllability problem was solved by K.-D. Phung in \cite{PHU} for equations set in a smooth domain of $\mathbb{R}^3$.
One difficult point comes from a consistency issue. To solve the Maxwell equations,  the source functions have to satisfy the local conservation of charge:
\begin{equation} \label{consmass}
\partial_t \tilde \rho + \nabla_x \cdot \tilde j =0, \quad \forall t \in [0,T], \, \forall x \in \T^2.
\end{equation}
K.-D. Phung bypasses this difficulty by showing that it is possible to find a control such that $\tilde \rho=0$ and $\operatorname{div} \tilde j =0$ and we follow this approach. 
We will actually use a slight variation on Phung's theorem.
Although we will not give a complete proof of this result, we provide for the sake of completeness several lemmas, which, when put together, yield the claimed result.

Using the celebrated H.U.M. (Hilbert Uniqueness Method) of J.-L. Lions, which consists of a duality argument, see \cite{LIO}, proving the controllability is equivalent to proving an observability inequality:

\begin{lem}
The Maxwell equations can be controlled in time $T$ if and only if the observability inequality on the adjoint system is satisfied: there exists $C>0$ such that for any $E_0,B_0 \in L^2(\T^2)$, with $\nabla \cdot E_0=0$, the solution $(E,B)$ to the homogeneous Maxwell system
\begin{equation} \label{Maxhomogene}
\left\{ \begin{split}
& \partial_t  B + c \operatorname{curl}  E = 0, \quad \partial_t  E + c \operatorname{curl}  B=0, \\
& \operatorname{div}  E  = 0, \quad \operatorname{div} B = 0, \\
&  E_{| t=0} = E_0, \quad  B_{| t=0} = B_0, \\
\end{split} \right.
\end{equation}
satisfies
\begin{equation*}
\| E \|^2_{L^2([0,T], L^2(\T^2))} + \| B \|^2_{L^2([0,T], L^2(\T^2))}  \leq C \int_0^T \int_\omega | E |^2 \, dx\, ds.
\end{equation*}
\end{lem}
Now we state another lemma due to Phung. This result actually comes from the conservation of the electromagnetic energy and a decomposition of the electromagnetic field using a suitable vector potential.

\begin{lem}
There exists $C_1>0$ such that  for any $E_0,B_0\in L^2(\T^2)$, with $\nabla \cdot E_0=0$, the solution $(E,B)$ to \eqref{Maxhomogene} satisfies:
\begin{equation*} 
\| B \|^2_{L^2([0,T], L^2(\T^2))} \leq C_1 \| E \|^2_{L^2([0,T], L^2(\T^2))}.
\end{equation*}
In particular, it is sufficient to prove observability for the electric field only.
\end{lem}

Then, writing the equation on $E$ as a wave equation, K.-D. Phung proved that the observability inequality is true in a domain with boundaries, using results of propagation of singularities due to L. H\"ormander, R. Melrose and J. Sj\"ostrand, as in the papers of J. Rauch and M. Taylor \cite{RT} and C. Bardos, G. Lebeau and J. Rauch \cite{BLR}. 
\begin{lem}
The observability inequality holds true if $T$ is large enough.
\end{lem}
%
%

In view of the control of the Vlasov-Maxwell system, we additionally need to build smoother controls, so that the associated electromagnetic field $(E,B)$ remains smooth (at least with Lipschitz regularity) for all times. This can be done by using the method of B. Dehman and G. Lebeau \cite{DL} (see also the paper of C. Bardos, G. Lebeau and J. Rauch \cite{BLR}) or by applying the systematic procedure proposed in the paper by S. Ervedoza and E. Zuazua \cite{EZ}. 
Gathering all results together, we obtain the following theorem.

\begin{thm}\label{controlMaxwell}
Assume that the control time $T>0$ is large enough. Let $k \in \N^*$.
For any $E_0, B_0, E_1,B_1 \in H^k(\T^2)$, with $\operatorname{div} E_0=\operatorname{div} E_1=0$, there exists a control function $ \tilde j \in \cap_{s=0}^kH^s([0,T]; H^{k-s}(\T^2))$, supported in $[0,T] \times \omega$, satisfying $j_{| t=0}=j_{| t=T}=0$, such that for all $t \in [0,T], \, \operatorname{div} \tilde j =0$ and such that the solution $(E,B)$ to the system:
\begin{equation*}
\left\{ \begin{split}
& \partial_t \tilde B + c \operatorname{curl} \tilde E = 0, \quad \partial_t \tilde E + c \operatorname{curl} \tilde B=-\tilde j, \\
& \operatorname{div} \tilde E  = 0,  \quad \operatorname{div} \tilde B = 0, \\
& \tilde E_{| t=0} = E_0, \quad \tilde B_{| t=0} = B_0,
\end{split} \right.
\end{equation*}
satisfies $\tilde E_{| t=T} = E_1, \quad \tilde B_{| t=T} = B_1$.
\end{thm}
\subsection{Construction of relevant reference solutions}
\label{SScons}
%
%
%
%
Let $\omega',\omega''$ be some smooth open sets, contained in $\omega$, which still satisfy GCC, such that $\overline{\omega'} \subset \omega'' \subset \overline{\omega''} \subset \omega$. 
Our aim is now to apply the controllability results for the Maxwell equation, in order to build, for some sufficiently large $T>0$, a suitable smooth reference solution $(\overline f, \overline E, \overline B)$ for the full Vlasov-Maxwell equations:
\begin{equation*}
\left\{ \begin{aligned}
& \partial_t \overline{f} + \hat{v}\cdot\nabla_x \overline{f}
 + \operatorname{div}_v \left[(\overline E + \frac{\hat{v}^\perp}{c} \overline B) \, \overline{f}\right]=\mathbbm{1}_\omega \overline{G}, \\
& \partial_t \overline E + c \operatorname{curl}\overline  B = -\int_{\R^2} \overline f \hat v \, dv,  \quad \partial_t \overline B + c \operatorname{curl} \overline E=0, \\
& \operatorname{div} \overline E  = \int_{\R^2} \overline{f} \, dv- \int_{\R^2\times \T^2} \overline{f} \, dv \, dx, 
\end{aligned} \right.
\end{equation*}
where $\overline{G}$ is a suitable source and which satisfies the following properties: 
\begin{align}
\label{Cond3.1}
&\overline f_{|t=0}=0, \, \overline E_{|t=0}=0, \, \overline B_{|t=0}=0 \text{ and } \overline f_{|t=T}=0, \, \overline E_{|t=T}=0, \, \overline B_{|t=T}=0. \\
\label{Cond3.2}
& \text{For all } t \in [0,T], v \in \R^2, \overline f(t,.,v) \text{ is compactly supported in } \omega. \\
\label{Cond3.3}
&\text{The characteristics } (X,V) \text{ associated with } \overline {E}+ \frac{\hat{v}^\perp}{c} \overline B \text{ satisfy the property:} \\
\nonumber
&\text{ for any } x\in \mathbb{T}^2, v \in \mathbb{R}^2, \text{ there is } t \in [T/9,8T/9] \text{ such that:}\\
\nonumber
& \hskip 5cm X(t,0,x,v) \in \omega', \quad \text{and  }|V|(t,0,x,v) \geq 4.
\end{align}

The main result of this section is the following:
\begin{prop}
\label{propGCC}
If $T>0$ is large enough with respect to $c$, there exists a reference solution $(\overline f, \overline E, \overline B)$ such that \eqref{Cond3.1}, \eqref{Cond3.2}, \eqref{Cond3.3} hold. 
\end{prop}
\begin{proof}[Proof of Proposition \ref{propGCC}]
%
To build such a reference solution, as explained in Section \ref{sec1}, we have to overcome the obstruction coming from slow particles.
The idea is to use an electric field to accelerate particles whose initial velocity is too low.

The elementary (yet crucial) observation is that on the torus, any constant electromagnetic field $(E_1,B_1) \in \R^2 \times \R$ is a stationary solution of the Maxwell equations (without source). 

 \par \ \par
 
\noindent {\bf 1.} In a first time interval $[0,T_1] \subset [0,T)$, we take as the reference solution 
$$(\overline{f},\overline{E},\overline{B})=(0,0,0),$$ 
which is of course a trivial stationary solution of Vlasov-Maxwell; by compactness, thanks to the geometric control condition satisfied by $\omega'$, there exists $T_1>0$ such that for all $x\in \mathbb{T}^d$, and all $v \in \R^2$ such that $|\hat{v}|>c/2$, 
\begin{equation} \label{3.1}
x + t \hat{v} \in \omega', \quad \text{for some  } t \in [0,T_1].
\end{equation}
Then there exists $m>4$, such that for any $|v|>m$, $|\hat{v}|>c/2$, which allows to say that \eqref{Cond3.3} is satisfied for all particles whose initial velocity $v$ has a modulus larger than $m$. 

 \par \ \par

\noindent {\bf 2.} There remains to take care of all other particles, whose velocity is not high enough. For that, the idea is to use  the control for the Maxwell equations (Theorem \ref{controlMaxwell}) to steer the electromagnetic field from $(0,0)$ to $(E_1,0)$, with $E_1:= (1,0)$, to wait for a sufficiently large time so that all particles get a sufficiently high velocity and meet the control zone as in the first step, and finally, to use once more the control for Maxwell to bring the electromagnetic field back to $(0,0)$. Let us now make this construction explicit.

Let $T_2>0$ be a large enough time such that we can apply Theorem \ref{controlMaxwell} and find a smooth $ \tilde j_2$ (with $ \nabla \cdot \tilde j_2=0$), supported in $\omega$ such that the following holds:
\begin{equation*}
\left\{ \begin{split}
& \partial_t \tilde B + c \operatorname{curl} \tilde E = 0, \quad \partial_t \tilde E + c \operatorname{curl} \tilde B=-\tilde j_2, \\
& \operatorname{div} \tilde E  = 0,  \quad \operatorname{div}  \tilde B = 0, \\
& \tilde E_{| t=T_1} = 0, \quad \tilde B_{| t=T_1} = 0, \\
& \tilde E_{| t=T_1+T_2} = E_1, \quad \tilde B_{| t=T_1+T_2} = 0.
\end{split} \right.
\end{equation*}
Now, we can consider two functions $\mathcal Z_i \in C^{\infty}_0(\R^{2} ; \R)$ (for $i=1$ or $2$)
satisfying the following constraints:
\begin{equation} \label{DefZ}
\left\{ \begin{array}{l}
{{\mathcal Z}_i \geq 0 \text{ in } \R^{2},} \\
{\mbox{Supp } {\mathcal Z}_i  \subset B(0,1), }\\
{\displaystyle{ \int_{\R^{2}} {\mathcal Z}_i (v) \, dv =0}, \quad \displaystyle{ \int_{\R^{2}} {\mathcal Z}_i(v) \, \hat{v} \, dv =(\delta_{i=1}, \delta_{i=2})}.   }
\end{array} \right.
\end{equation}
We then set 
$$\overline f = \mathcal{Z}_1 (v)  \tilde{j}_{2}^{1}+\mathcal{Z}_2 (v)  \tilde{j}_{2}^{2}, \quad \text{   where   } \tilde{j}_2 =( \tilde{j}_{2}^{1}, \tilde{j}_{2}^{2}),$$
on $[T_1,T_1+T_2]$. We take $(\overline{E}, \overline{B})=(\tilde{E}, \tilde{B})$ and we notice that $(\overline f, \overline E, \overline B)$ is a solution of the Vlasov-Maxwell system on $[T_1,T_1+T_2]$, with a suitable source (supported in $\omega$): 
\begin{equation*}
\overline{G}(t,x,v):= \partial_{t} \overline{f} + v \cdot \nabla_{x}
\overline{f} + \left(\tilde E + \frac{\hat{v}^\perp}{c} \tilde B\right) \cdot \nabla_{v} \overline{f}.
\end{equation*}
Notice that $\overline{G}$ is compactly supported in $\omega$, by construction of $\tilde{j}_2$ and $\overline{f}$.

The effect of constant the force field $E_1$ is to accelerate particles (even if it takes a very long time): denoting by $(X,V)$ the characteristics associated with $E_1$, we can prove the existence of some $T_3>0$ such that for any $(x,v) \in \T^2 \times \R^2$, there exists $ t \in [T_1+T_2, T_1 + T_2+ T_3]$, such that:
\begin{equation} \label{3.2}
\forall x \in \T^2, \, \forall v \in \R^2 \text{  s.t.  } |v|<m, \quad X(t,0,x,v) \in \omega'  \quad \text{and  }|V|(t,0,x,v) \geq 4.
\end{equation}
%

Then, on the interval of time $[T_1+T_2, T_1+ T_2+T_3]$, we consider $(\overline f,\overline{E}, \overline{B})=(0,E_1,0)$ as the reference function, which is  a stationary solution of Vlasov-Maxwell without source.

 \par \ \par

\noindent {\bf 3.} Finally, one finds a smooth $ \tilde j_3$ such that the following holds:
\begin{equation*}
\left\{ \begin{split}
& \partial_t \tilde B + c \operatorname{curl} \tilde E = 0, \quad \partial_t \tilde E + c \operatorname{curl} \tilde B=-\tilde j_2, \\
& \operatorname{div} \tilde E  = 0, \quad \operatorname{div}  \tilde B = 0, \\
& \tilde E_{| t=T_1+T_2+T_3} = E_1, \quad \tilde B_{| t=T_1+T_2+T_3} = 0, \\
& \tilde E_{| t=T_1+2T_2+T_3} = 0, \quad \tilde B_{| t=T_1+2T_2+T_3} = 0.
\end{split} \right.
\end{equation*}
and we can then define
$$\overline f = \mathcal{Z}_1 (v)  \tilde{j}_{3}^{1}+\mathcal{Z}_2 (v)  \tilde{j}_{3}^{2}, \quad \text{   where   } \tilde{j}_3 =( \tilde{j}_{3}^{1}, \tilde{j}_{3}^{2}),$$
on $[T_1+T_2+T_3,T_1 +2T_2+T_3]$. As before, we take $(\overline{E}, \overline{B})=(\tilde{E}, \tilde{B})$. Finally we choose $T>T_1 +2T_2+T_3$. By construction, \eqref{Cond3.1} and \eqref{Cond3.2} are satisfied. We observe finally that \eqref{Cond3.3} is satisfied for some times in $[0,T]$, because of \eqref{3.1} and \eqref{3.2}. Obtaining the time interval $[T/9,8T/9]$ in \eqref{Cond3.3} is a matter of translating the solution in time and extending it by $0$ when necessary.

\end{proof}
\section{The reference solution under the strip assumption} 
\label{sec3}
In this section, we build a reference solution $(\overline{f},\overline{E},\overline{B})$ for the Vlasov-Maxwell system, in the case where we make the strip assumption on the control set. We shall also assume that the speed of light satisfies $c\geq c_0$ (where $c_0>0$ will be fixed later, see Lemma \ref{PropSolRefCM}).
This construction is central in the proofs of Theorems \ref{Theo2} and \ref{Theo}. 

This section is divided into three sub-sections. First, we build a relevant reference solution for a relativistic Vlasov-Poisson system. Then, we state an approximation lemma for the Maxwell equations (as $c\rightarrow + \infty$), which we use in the last subsection to build a reference solution for the Vlasov-Maxwell system.

\subsection{The case of Vlasov-Poisson with an external magnetic field } 
Let $x_0 \in \omega$ and $r_0>0$ small enough such that $B(x_0,2r_0)\subset \omega$. Let $\mathfrak{b}$ some nonzero {\it constant} external magnetic field. 
Let $R$ be a large enough positive constant; we construct a reference solution $\overline{f}$, which depends on $R$, and which will be used to solve the controllability problem for data $f_0$ and $f_1$ supported in $\T^2 \times B(0,R/2)$.

We first aim at finding a reference solution $\overline{f}$  for the \emph{relativistic} Vlasov-Poisson equation, with the external magnetic field $\mathfrak{b}$:
\begin{equation} \label{VlasovPoisson}
\left\{ \begin{aligned}
& \partial_t \overline{f} + \hat{v}\cdot\nabla_x \overline{f}
+ \operatorname{div}_v \left[(\nabla_x \overline\varphi + {\hat{v}^\perp} \mathfrak{b}) \, \overline{f}\right]=\mathbbm{1}_\omega G, \\
& \Delta_x \overline\varphi = \overline\rho - 1, \\
\end{aligned} \right.
\end{equation}
where $G$ is a suitable source and which satisfies the following properties, for some large enough $T>0$:
\begin{align}
\label{Cond4.1}
&\overline f_{|t=0}=\overline f_{|t=T}=0. \\
\label{Cond4.2}
&\forall t\in [0,T], \forall v \in \R^2, \,  \overline f(t,.,v) \text{ is compactly supported in } \omega. \\
\label{Cond4.3}
&\text{The characteristics } (X,V) \text{ associated with } \nabla {\overline  \varphi}+ {\hat v^\perp}  \mathfrak{b} \text{ satisfy the property:} \\
\nonumber
&\text{for any } x\in \mathbb{T}^2, v \in B(0,R), \text{ there is } t \in [T/9,8T/9] \text{ such that }  \\
\nonumber
& \hskip 5cm X(t,0,x,v) \in B(x_0,r_0/2),  \quad \text{and  } |V|(t,0,x,v) \geq 5.
\end{align}
Building such a reference solution has been already done in \cite{GHK} for the \emph{classical} Vlasov-Poisson equation; actually it was done for any arbitrarily small control time $T>0$, but the generalization to this relativistic version of Vlasov-Poisson imposes to take a large enough control time. Furthermore, having in mind an application to the Vlasov-Maxwell equation, because of  consistency issues, we have to impose two additional conditions. First that the local conservation of charge is satisfied \emph{everywhere},
\begin{equation} \label{LCC}
\partial_t \overline\rho + \div \overline j =0, \quad \forall x \in \T^2, \ \forall t \in [0,T],
\end{equation}
(where $\overline \rho := \int \overline f \, dv$ and $\overline j = \int \overline f \hat{v} \, dv$) and second that the ``zero-mean current'' condition  holds
\begin{equation} \label{ZM}
\int_{\T^2}  \overline j \, dx =0,  \ \forall t \in [0,T].
\end{equation}

Let us introduce some notations before getting to the construction of $\overline{f}$. \par \ \par

\noindent {\bf Notation.} 
 We denote by $\mathcal{D}$ a line of $\T^2$, which does not cut the control zone $\omega$ (reduce $\omega$ if necessary, but in a way that it still satisfies the strip assumption) and $n$ a unit vector, orthogonal to $\mathcal{D}$.
 
 We also recall that $\mathcal{H}_{2d}$ is defined in \eqref{H2d}.  \par \ \par

We have the following lemmas, which are taken from \cite{GHK}, with modifications to ensure that \eqref{LCC} and \eqref{ZM} are satisfied.
The first lemma states that we can find some relevant electric field, whose effect will be to accelerate particles. \par
\begin{lem}
\label{PropAccelerePartout-mag} 
Let $\mathfrak{b} \in \R$ be given.
Let $\tau>0$ and $M>0$. There exists $\tilde{M}>0$, ${\mathcal E} \in C^{\infty}([0,\tau] \times \T^2;\R^{2})$ and  $\varphi \in C^{\infty}([0,\tau] \times \T^2;\R)$ which do not depend on $c$, satisfying
\begin{gather}
\label{EetPhi-mag}
{\mathcal E} = -\nabla \varphi \text{ in  } [0,\tau] \times (\T^2 \setminus \mathcal{H}_d), \\
\label{Phi2SupportTemps-mag}
\mbox{Supp}({\mathcal E}) \subset (0,\tau) \times \T^2, \\
\label{Phi2Harmonique-mag}
\Delta \varphi =0 \text{ in  } [0,\tau] \times (\T^2\setminus \mathcal{H}_d), \\
\label{supp}
\int_{\mathcal{D}} \nabla \varphi \cdot n \, dx = 0,
\end{gather}
such that, for any $\mathfrak{F} \in L^{\infty}(0,T;W^{1,\infty}(\T^{2} \times \R^{2}))$ satisfying $\| \mathfrak{F} \|_{L^{\infty}} \leq 1$,  if $({X},{V})$ are the characteristics corresponding the force $\mathfrak{F} + {\mathcal E} + {\hat v^\perp}  \mathfrak{b}$, then
\begin{equation} \label{AccelerePartout-Mag}
\forall (x,v) \in \T^{2} \times B(0,M), \ 
{V}(\tau,0,x,v) \in B(0,\tilde{M}) \setminus B(0,M+1).
\end{equation}
\end{lem}
The proof is close to that of  \cite[Proposition 5.2]{GHK}. The difference with the statement of \cite{GHK} appears in \eqref{supp}, which will be crucial to ensure \eqref{LCC}. To prove all requirements, one has to find $w \in C^\infty(\T^2;\R^2)$ satisfying~:
\begin{gather}
\label{NulRot}
\rot w = 0 \text{ in } \T^2 \setminus\mathcal{H}_d, \\
\label{NulDiv}
\div w = 0 \text{ in } \T^2 \setminus \mathcal{H}_d, \\
\label{VPasNul}
|w(x) | > 0 \mbox{ for any } x \mbox{ in } \T^2 \setminus \mathcal{H}_d, \\
\label{CirculationFluxNuls}
\int_{\mathcal{D}} w \cdot n \, dx= \int_{\mathcal{D}} w\cdot d\tau =0.
\end{gather}
Such a function was constructed in  \cite[Appendix A.2]{GLA}. Once it is obtained, one can take $\mathcal{E} := w$ and observe that it coincides with a gradient outside of $\mathcal{H}_d$.

Note in particular that the solution here does not correspond to the one giving rise to the force field $n$ (which is the ``most natural'' one could think of; it satisfies \eqref{NulRot}, \eqref{NulDiv} and \eqref{VPasNul} but not \eqref{CirculationFluxNuls}). The effect of this force field $\mathcal{E}$ is only to accelerate particles, and it is very unlikely that it allows to ``cure'' the bad directions. 

For that issue, we will fully rely on a ``bending'' effect of the magnetic field $\mathfrak{b}$, which is described in the next lemma.

\begin{lem} \label{PropSolRefCM}
Let $\overline{M}>1$. Given $\mathfrak{b} \neq 0$, there exist:
\begin{itemize}
\item $c_{0}>0$ depending on $\mathfrak{b}, x_0, r_0$, 
\item $\underline{m} >0$ depending only on $\mathfrak{b}, x_0, r_0$, 
\item$T_0>0$ depending on $\mathfrak{b}$, $x_0, r_0$ and $\overline{M}$, and 
\item $\kappa$ depending on $\mathfrak{b}$, $x_0, r_0$ and $\overline{M}$,
\end{itemize}
such that for all $\mathfrak{F} \in L^{\infty}(0,T;W^{1,\infty}(\mathbb{T}^{2} \times \mathbb{R}^{2}))$ satisfying $\| \mathfrak{F} \|_{L^{\infty}} \leq \kappa$, if $c \geq c_{0}$ then the characteristics $(\overline{X}, \overline{V})$ associated with ${\hat{v}^\perp} \mathfrak{b} +\mathfrak{F}$ satisfy:
\begin{multline} \label{GVCM}
\forall x \in \mathbb{T}^2, \forall v \in \mathbb{R} ^2 \text{ such that }  \overline{M} \geq | v | \geq \underline{m},
\exists t \in (T_0/4,3T_0/4), \ \overline{X}(t,0,x,v) \in B(x_0,r_0/2) \\
 \text{ and for all } s \in [0,T_0], \ \ \frac{|v|}{2} \leq |\overline{V}(s,0,x,v)| \leq 2|v|.
\end{multline}
\end{lem}

\


The proof of Lemma \ref{PropSolRefCM} (actually, of a generalized version of Lemma \ref{PropSolRefCM}) is given in the appendix.

With these two lemmas in hand, we now have:
\begin{prop} \label{VPext}
There exists $T>0$ large enough, such that there exists a reference solution $(\overline f, \overline \varphi)$  such that \eqref{Cond4.1}, \eqref{Cond4.2}, \eqref{Cond4.3}, \eqref{LCC}, \eqref{ZM} are satisfied.
In addition, $\overline{\rho}:= \int f \, dv $ and $\overline{j} := \int f \hat{v} \, dv$ do not depend on $c$.
\end{prop}
\begin{proof}[Proof of Proposition \ref{VPext}] The proof is divided into three steps. \par \ \par
%

\noindent {\bf 1.} First, we define the reference potential $\overline{\phi}:[0,T] \times \T^{2} \rightarrow \R$ as follows. 
We apply Lemma \ref{PropSolRefCM} for $\overline{M}=R+1$, and we obtain some time $T_0$ and some $\underline{m}>10$ such that \eqref{GVCM} is satisfied. Then we apply Lemma \ref{PropAccelerePartout-mag} with $\tau =T_0$ and
\begin{equation} \label{DefMHorrible}
M= \max \Big( \underline{m} + 1, 32 r_{0} (|\mathfrak{b}|+1) \Big),
\end{equation}
and obtain some $\overline{\phi}_{2}$, $\overline{{\mathcal E}}_{2}$ and some $\tilde{M}>0$ such that \eqref{AccelerePartout-Mag} is satisfied. We apply Lemma \ref{PropSolRefCM} for $\overline{M}=\tilde{M}$, and we obtain some time $T_1$ (with the same $\underline{m}>0$). We set
\begin{equation*}
\overline{\phi}(t,\cdot) = 
\left\{ \begin{array}{l}
 0 \text{ for } t \in [0,T_0] \cup [2T_0,2T_0 +T_1], \\
 \overline{\phi}_{2}(t - T_0,\cdot) \text{ for } t \in [T_0,2T_0],
\end{array} \right.
\end{equation*}
and
\begin{equation*}
\overline{{\mathcal E}}(t,\cdot) = 
\left\{ \begin{array}{l}
 0 \text{ for } t \in [0, T_0] \cup [2T_0,2T_0+T_1], \\
 \overline{{\mathcal E}}_{2}(t - T_0,\cdot) \text{ for } t \in [T_0,2T_0].
\end{array} \right.
\end{equation*}

We define $T := 2T_0 +T_1$. \par \ \par

\noindent {\bf 2.} Let us now introduce a first distribution function $\tilde{f}$.
Consider a function $\mathcal Z \in C^{\infty}_0(\R^{2} ; \R)$
satisfying the following constraints:
\begin{equation} \label{DefZ1}
\left\{ \begin{array}{l}
{{\mathcal Z} \geq 0 \text{ in } \R^{2},} \\
{\mbox{Supp } {\mathcal Z} \subset B_{\R^{2}}(0,1), }\\
{\displaystyle{ \int_{\R^{2}} {\mathcal Z}(v) \, dv =1}, \quad \displaystyle{ \int_{\R^{2}} {\mathcal Z}(v) \, \hat{v} \, dv =0}.   }
\end{array} \right.
\end{equation}
We introduce $\tilde{f}=\tilde{f}(t,x,v)$ as
\begin{equation} \label{Deffbar}
\tilde{f}(t,x,v):= {\mathcal Z}(v) \Delta \overline{\varphi}(t,x).
\end{equation}
Of course, $\tilde{f}$ satisfies \eqref{VlasovPoisson}
in $[0,T] \times\T^{2} \times \R^{2}$, with source term
\begin{equation} \label{DefGbar}
\overline{G}(t,x,v):= \partial_{t} \overline{f} + \hat{v}\cdot\nabla_{x}
\overline{f} + \left(\nabla \overline{\varphi}+ \mathfrak{b}{\hat{v}^\perp}\right)\cdot\nabla_{v} \overline{f},
\end{equation}
which is supported in $[0,T] \times B(x_{0},r_{0}) \times \R^{2}$. 
Up to an additive function of $t$, the function $\overline\varphi$
satisfies the Poisson equation corresponding to $\tilde{f}$. We can observe that $\tilde{f}$ also satisfies:
\begin{equation*}
\tilde{f}_{| t=0}= 0, \quad \tilde{f}_{| t=T}=0.
\end{equation*}
%
%
Finally, we denote 
\begin{equation} \label{DefOvRho}
\overline{\rho}(t,x):=\int_{\R^{2}} \overline{f}(t,x,v) \, dv =\Delta \overline{\varphi}(t,x). 
\end{equation}
By construction, according to Lemmas \ref{PropAccelerePartout-mag} and \ref{PropSolRefCM}, \eqref{Cond4.1}, \eqref{Cond4.2} and \eqref{Cond4.3} hold. \par \ \par

\noindent {\bf 3.} Consequently, up to now, the solution $(\tilde f, \overline \varphi)$ satisfies all requirements but \eqref{LCC} and \eqref{ZM}. During the interval of time $[0,T_0] \cup [2T_0,2T_0 + T_1]$, these are trivially satisfied since $\tilde f=0$.
The problem appears only for the interval of time $[T_0, 2 T_0]$, where in general, we do not have the compatibility conditions:
\begin{equation*}
\partial_t \int_{\R^2} \tilde f \, dv + \nabla_x \cdot \int_{\R^2} \tilde f \hat v \, dv = 0.
\end{equation*}
and
\begin{equation*}
\int_{\mathbb{T}^2} \int_{\R^2} \tilde{f} \hat{v} \, dv \, dx = 0.
\end{equation*}
To ensure these conservation laws, the idea is add a correction to $\tilde{f}$. Let 
$$\overline g := u_1(t,x) \mathcal{Z}_1(v)+ u_2(t,x) \mathcal{Z}_2(v),$$ 
with $u(t,\cdot):=(u_1,u_2)$ compactly supported in an open set $U$ such that $\overline{U} \subset \omega$, where we recall that $\mathcal{Z}_i$ (for $i=1,2$) is defined by:
%
\begin{equation*}
\left\{ \begin{array}{l}
{{\mathcal Z}_i \geq 0 \text{ in } \R^{2},} \\
{\mbox{Supp } {\mathcal Z}_i  \subset B(0,1), }\\
{\displaystyle{ \int_{\R^{2}} {\mathcal Z}_i (v) \, dv =0}, 
	\quad \displaystyle{ \int_{\R^{2}} {\mathcal Z}_i(v) \, \hat{v} \, dv =(\delta_{i=1}, \delta_{i=2})}.}
\end{array} \right.
\end{equation*}
Let us show that we can choose $u$ such that:
\begin{equation*}
\partial_t \int_{\R^2} \tilde f \, dv  = - \nabla_x \cdot \int_{\R^2} \overline g v \, dv,
\end{equation*}
and 
$$
\int_{\T^2 \times \R^2} \overline{g} \,\hat{v} \, dv \, dx =0.
$$

For simplicity, we denote $h := - \partial_t \int \overline f \, dv $. The problem is equivalent to finding a vector field $u \in C^\infty(\R^+ \times\T^2)$, satisfying all four conditions:
\begin{eqnarray} 
\label{Eq:a.}
 u_{|t=0}= u_{|t=T}=0, \\
\label{Eq:b.}
\div u = h, \quad \forall t \in [0,T], x \in \T^2, \\
\label{Eq:c.}
\forall t\in [0,T],  \mbox{Supp\,} u(t,\cdot) \subset \omega, \\
\label{Eq:d.}
\int_{\T^2} u \, dx =0, \quad \forall t \in [0,T].
\end{eqnarray}
To build such a function $u$, we introduce $\theta \in C^\infty([0,T] \times\T^2)$ a solution of the elliptic equation:
\begin{equation*}
\Delta \theta =h, \quad \text{for } \ t \in [0,T], \  x \in \T^2.
\end{equation*}
Observe that $\int_{\T^{2}} h \, dx = 0$ due to \eqref{DefOvRho}. 
With the notations of Lemma \ref{PropAccelerePartout-mag}, we notice that we can take
$$\theta = -\partial_t \varphi.$$
We thus conclude that $\nabla \theta$ is divergence free outside $\omega$. Therefore, according to the Hodge-Poincar\'e lemma (recall that $\omega$ contains a hyperplane of the torus with normal vector $n$), there exists  $\Psi \in  C^\infty([0,T] \times\T^2)$ and $\alpha\in \R$ such that:
\begin{equation*}
\nabla \theta = \nabla^\perp \Psi + \alpha n \text{ in } \T^{2} \setminus \omega,
\end{equation*}
where $\alpha = \int_{\mathcal D} \nabla \theta \cdot n \, dx$. But the reference potential $\varphi$ defined above is designed as to satisfy $\int_{\mathcal D} \nabla \varphi \cdot n \, dx=0$ (see \eqref{supp} in Lemma \ref{PropAccelerePartout-mag}). This yields that $\alpha = 0$. Now, we introduce a smooth cut-off function $\eta$ such that $\eta \equiv 0$ on $U$, and $\eta \equiv 1$ on $\T^2 \setminus \omega$. 
We define a vector field $u$ as:
\begin{equation*}
u = \nabla \theta - \nabla^\perp( \eta \Psi),
\end{equation*}
and one can finally check that $u$ satisfies the four above conditions \eqref{Eq:a.}, \eqref{Eq:b.}, \eqref{Eq:c.} and \eqref{Eq:d.}.

We can now set:
\begin{equation*}
\overline f = \tilde f + \overline g.
\end{equation*}
Note that by construction, we have:
$$\int_{\R^2} \overline f \, dv = \int_{\R^2} \tilde f \, dv,$$ 
so that the electric field created by $\overline f$ (through the Poisson equation) is the same as the one created by $\tilde f$, and that 
$$\int \overline f \, \hat{v} \, dv = \int \overline g \, \hat{v} \, dv.$$ 
One can readily check that $\overline{f}$ is a solution of the Vlasov-Poisson system with a suitable source in $\omega$, and that \eqref{Cond4.1}, \eqref{Cond4.2}, \eqref{Cond4.3}, \eqref{LCC}, \eqref{ZM} are satisfied (up to shifting and rescaling the time interval $[0,T]$). The fact that $\overline{\rho}$ and $\overline{j}$ do not depend on $c$ (but only on $c_{0}$) is due to the fact that the constructions of Lemmas \ref{PropAccelerePartout-mag} and  \ref{PropSolRefCM} do not depend on $c$ (but only on $c_0$).
\end{proof}
\subsection{The approximation lemma}
Now, we want to use the following approximation lemma, which is valid for large values of the speed of light.
The goal is to approximate the Vlasov-Maxwell dynamics by the simpler one of Vlasov-Poisson with an external magnetic field. \par
For any function $A: \T^d \rightarrow \R$, $d=2$ or $3$, at least in $L^1$, we denote by $\hat A^{k}$ the k-th Fourier coefficient of $A$, for $k \in \Z^d$.
\begin{lem} \label{approx}
Let $d= 2$ or $3$. Let $E_0,B_0, j, \rho$ some $C^\infty$ functions. Let $(E,B)$ the solution  to the Maxwell equations:
\begin{equation*}
\left\{ \begin{aligned}
& \partial_t B + c \operatorname{curl} E = 0, \quad \partial_t E + c \operatorname{curl} B=-j \ \text{ in } [0,T] \times \T^{d}, \\
& \operatorname{div} E  = \rho - \int \rho \, dx,  \quad \operatorname{div} B = 0  \ \text{ in } [0,T] \times \T^{d} , \\
& E_{| t=0} = E_0, \quad B_{| t=0} = B_0  \ \text{ in } \T^{d}.
\end{aligned} \right. 
\end{equation*}
%
Assume that at initial time the compatibility conditions are satisfied:
\begin{align*}
& \operatorname{div} E_0 = \rho(0) - \int \rho(0) \, dx  \ \text{ in } [0,T] \times \T^{d}, \\ 
& \operatorname{div} B_0 =0  \ \text{ in }  \T^{d}.
\end{align*}
 Assume that the local conservation of charge is satisfied for all times:
\begin{equation*}
\forall x \in \T^d, \quad \partial_t \rho + \nabla_x\cdot j =0,
\end{equation*}
as well as the zero-mean current property:
\begin{equation} \label{ZMcond}
 \int_{\T^2} j \, dx = 0.
\end{equation}
Let $E_\infty$ be the solution of the Poisson equation:
\begin{equation*}
\left\{ \begin{aligned}
& \operatorname{curl} E_\infty=0  \ \text{ in } [0,T] \times \T^{d}, \\
& \operatorname{div} E_\infty  = \rho- \int \rho \, dx  \ \text{ in } [0,T] \times \T^{d}.
\end{aligned} \right.
\end{equation*} 
Then, we have for all $t>0$:
\begin{equation} \label{eqEstimB}
\left\| B - \int_{\T^{d}} B_0 \, dx - \tilde{B} \right\|_{L^\infty([0,t] \times \mathbb{T}^d)} \leq \frac{C_{\rho,j} (t+1)}{c},
\end{equation}
and
\begin{equation} \label{eqEstimE}
\left\| E- E_\infty - \int_{\T^{d}} E_0 \, dx - \tilde{E} \right\|_{L^\infty([0,t] \times \mathbb{T}^d)} \leq \frac{C'_{\rho,j} (t+1)}{c},
\end{equation} 
where  $C_{\rho,j}$ and $C'_{\rho,j}$ are constants depending only on $\rho$ and $j$ and where $\tilde{B}$ and $\tilde{E}$ are defined by their space Fourier coefficients:
\begin{equation} \label{DefLesTildes}
\left\{ \begin{array}{l}	
	\widehat{\tilde{B}}^{0}=0, \\
	\widehat{\tilde{B}}^{k} = \frac{k \wedge \hat{E}^{k}(0)}{|k|} \sin (tc|k|) + \hat{B}^{k}(0) \cos (tc|k|) \text{ for } k \neq 0, \\
	\widehat{\tilde{E}}^{0}=0, \\
	\widehat{\tilde{E}}^{k} = - \frac{k \wedge \hat{B}^{k}(0)}{|k|} \sin (tc|k|) 
	+ \Big[\hat{E}^{k}(0) - \frac{ik}{|k|^{2}}\hat{\rho}^{k}(0) \Big] \cos (tc|k|) \text{ for } k \neq 0.
	\end{array} \right.
\end{equation}
\end{lem} 
\begin{rema}
The initial data in this lemma are somehow ill-prepared. In contrast, in the well-prepared case, which would correspond to the additional assumptions:
$$
 \operatorname{curl} E_0=0, \quad  \operatorname{curl} B_0=0,
$$
then we would have $\tilde{E}=0$ and $\tilde{B}=0$ and this would yield strong convergence results in $L^\infty([0,t] \times \T^d)$. In the present case, the convergence is only weak in time.
\end{rema}
\begin{proof}[Proof of Lemma \ref{approx}]
The electromagnetic field $E$ and $B$ satisfies the wave equations:
\begin{equation*}
\left\{ \begin{aligned}
& \partial^2_t E - c^2\Delta_x E=-c^2 \nabla \rho-\partial_t j, \\
& \partial^2_t B - c^2\Delta_x B=c \operatorname{curl} j.
\end{aligned} \right.
\end{equation*}
Thus we have:
\begin{equation*}
\left\{ \begin{aligned}
& \partial^2_t \hat E^{k} + c^2 | k |^2 \hat E^{k} =-ikc^2 \hat \rho^{k} - \partial_t \hat j^{k}, \\
& \partial^2_t \hat B^{k} + c^2 | k |^2 \hat B^{k}= c ik \wedge \hat j^{k},
\end{aligned} \right.
\end{equation*}
 that we easily solve, for $k \neq 0$, with Duhamel's formula:
 \begin{align}
 \hat E^{k} = \frac{1}{c | k |} \partial_t \hat E^{k} (0) \sin(t c | k |) + \hat E^{k}(0) \cos(tc | k |) + \int_0^t [-c^2 ik \hat \rho^{k}(s)-\partial_t \hat j^{k}] \frac{\sin((t-s)c | k |)}{c | k |}\, ds, \\
 \hat B^{k}= \frac{1}{c | k |} \partial_t \hat B^{k} (0) \sin(t c | k |) + \hat B^{k}(0) \cos(tc | k |) + \int_0^t  c ik \wedge \hat j^{k}(s) \frac{\sin((t-s)c | k |)}{c | k |}\, ds .
 \end{align}

Let us first deal with $\hat E^{k}$. We have, by integration by parts:
\begin{multline*}
 - \int_0^t c^2 ik \hat \rho^{k}(s) \frac{\sin((t-s)c | k |)}{c | k |}\, ds \\
 = \frac{-ik}{| k |^2}  \hat \rho^{k}(t) + \frac{ik}{| k |^2} \hat \rho^{k}(0) \cos (tc | k |)+ \int_0^t  \frac{ik}{| k |^2} \partial_s \hat \rho^{k}(s) \cos((t-s) c | k |)\, ds.
 \end{multline*}
 
We remark that:
 \[
 \frac{-ik}{| k |^2}  \hat \rho^{k}(t) = \hat E^{k}_\infty(t).
 \]
By another integration by parts, we get:
\begin{multline*}
\int_0^t  \frac{ik}{| k |^2} \partial_s \hat \rho^{k}(s) \cos((t-s) c | k |)\, ds \\
=  \frac{ik}{c | k |^3} \partial_t \hat \rho^{k} (0) \sin (tc | k|) + \int_0^t  \frac{ik}{c | k |^3} \partial_{ss} \hat \rho^{k}(s) \sin((t-s) c | k |)\, ds.
 \end{multline*}

Moreover, we have:
\[
\left|  \int_0^t  \frac{ik}{c | k |^3} \partial_{ss} \hat \rho^{k}(s) \sin((t-s) c | k |)\, ds \right| \leq \frac{1}{c}   \frac{t}{| k |^2} \Vert \partial_{tt} \hat \rho^{k} \Vert_{L^\infty_t}.
\]
Finally, in a straightforward manner, we have
\[
\left| \int_0^t \partial_t \hat j^{k} \frac{\sin((t-s)c | k |)}{c | k |}\, ds  \right| \leq \frac{t \Vert \partial_t \hat j^{k}\Vert_{L^\infty_t}}{c | k|}.
\]
so that for $k\neq 0$:
\begin{multline}
\Big| \hat E^{k}(t) - \hat E_\infty^{k}(t)  - \frac{k \wedge \hat{B}^{k}(0)}{|k|} \sin (tc|k|) 
		+ \Big[\hat{E}^{k}(0) - \frac{ik}{|k|^{2}}\hat{\rho}^{k}(0) \Big] \cos (tc|k|) \Big| \\
\leq \frac{t}{c| k |^2} \| \partial_{tt} \hat \rho^{k} \|_{L^\infty_t}
+ \frac{1}{c}   \left( \frac{\| \partial_{t} \hat \rho^{k} \|_{L^\infty_t}}{| k |^2} + \frac{\| \hat{j}^{k} \|_{L^\infty_t}}{|k|}\right)
+ \frac{t \| \partial_t \hat j^{k}\|_{L^\infty_t}}{c | k|}.
\end{multline}

For $k=0$, one has, using the zero-mean current condition \eqref{ZMcond}, that 

\begin{equation*}
\hat{E}^{0}(t) - \hat{E}_\infty^0(t) = \hat{E}^{0}(0) .
\end{equation*}

Finally we have:

\begin{multline*}
\left\| E- E_\infty -\int_{\T^{d}} E_0 \, dx - \tilde{E} \right\|_{L^\infty([0,t] \times \mathbb{T}^2)}  \\
\leq \frac{1}{c}   \sum_{k\neq 0} \frac{t}{| k |^2} \Vert \partial_{tt} \hat \rho^{k} \Vert_{L^\infty_t} 
+ \left( \frac{\| \partial_{t} \hat \rho^{k} \|_{L^\infty_t}}{| k |^2} + \frac{\| \hat{j}^{k} \|_{L^\infty_t}}{|k|}\right)
+ \frac{t}{| k|}\Vert \partial_t \hat j^{k}\Vert_{L^\infty_t}.
\end{multline*}
The summability is a simple consequence of the smoothness of the sources $\rho$ and $j$.
Hence, the estimation (\ref{eqEstimE}) holds with 
$$C'_{\rho,j}=   \sum_{k\neq 0} \frac{1}{| k |^2} \Vert \partial_{tt} \hat \rho^{k} \Vert_{L^\infty_t}
+ \frac{1}{c}   \left( \frac{\| \partial_{t} \hat \rho^{k} \|_{L^\infty_t}}{| k |^2} + \frac{\| \hat{j}^{k} \|_{L^\infty_t}}{|k|}\right)
+\frac{1 }{ | k|}\Vert \partial_t \hat j^{k}\Vert_{L^\infty_t}.$$
The magnetic field $B$ is treated likewise.
\end{proof}
\subsection{The reference solution} 
We now rely on the approximation lemma to find a suitable reference solution for the Vlasov-Maxwell system, provided that the speed of light is large enough, and that the initial magnetic field $B_0$ is of the form $c \, b_0$ where $b_0$ satisfies 
\begin{equation*}
\int_{\T^{2}} b_{0} \, dx \neq 0.
\end{equation*}
More precisely, we look for $c$ large enough such that there exists a smooth reference solution $(\overline{f}, \overline{E},\overline{B})$ to the Vlasov-Maxwell equation:
\begin{equation} \label{VlasovMax}
\left\{ \begin{aligned}
& \partial_t \overline{f} + \hat{v}\cdot\nabla_x \overline{f}+  \operatorname{div}_v \left[(E + \frac{\hat{v}^\perp}{c} \overline{B}) \, \overline{f}\right]=\mathbbm{1}_\omega \overline{G} , \\
& \partial_t \overline{E}_1 - c \partial_{x_2} \overline{B}=-\int_{\R^2} \overline{f}  \hat v_1 \, dv, \quad \partial_t \overline{E}_2 + c \partial_{x_1} \overline{B}=-\int_{\R^2} \overline{f}  \hat v_2 \, dv,  \\
& \partial_t \overline{B} +  \partial_{x_1} \overline{E}_2 -  \partial_{x_2} \overline{E}_1=0, \\
& \partial_{x_1} \overline{E}_1 +  \partial_{x_2} \overline{E}_2 =  \int_{\R^2} \overline{f}  \, dv- \int_{\T^2 \times \R^2} \overline{f}  \, dv \, dx , \\
\end{aligned} \right.
\end{equation} 
with $\overline{G}$ is a suitable source, which satisfies the following properties: 
\begin{align}
\label{Cond5.1}
& \overline f|_{t=0}=0, \ \overline f|_{t=T}=0, \ \overline E|_{t=0}=0 \text{ and } \overline{B}|_{t=0}=B_0, \\
\label{Cond5.2}
&\forall t\in [0,T], \ \forall v \in \R^{2}, \ \overline f(t,.,v) \text{ is compactly supported in } \omega, \\
\label{Cond5.3}
& \text{The characteristics } (X,V) \text{ associated with } \overline{E}+ \frac{\hat{v}^\perp}{c} \overline{B} \text{ satisfy the property:} \\
\nonumber
& \text{for any } x\in \mathbb{T}^2, v \in \mathbb{R}^2, \text{ there is } t \in [T/10,9T/10] \text{ such that } \\
\nonumber
& \hskip 5cm X(t,0,x,v) \in B(x_0,r_0) \quad \text{and  } |V|(t,0,x,v)\geq 4.
\end{align}

\begin{prop} \label{refVM}
Let $T>0$ large enough. There exists $c(T)>0$ such that for any $c>c(T)$, there exists a reference solution $(\overline f, \overline E, \overline B)$ of \eqref{VlasovMax} such that \eqref{Cond5.1}, \eqref{Cond5.2} and \eqref{Cond5.3} hold. 
\end{prop}
\begin{proof}[Proof of Proposition \ref{refVM}]
Let $(\overline{f},\nabla \overline{\varphi})$ the reference function in Proposition \ref{VPext}, corresponding to the case of Vlasov-Poisson with external magnetic field, relevant for some large enough $T>0$. Let $(\overline E,\overline B)$ the electromagnetic field satisfying Maxwell equations with $\rho= \int \overline f \, dv$ and $j= \int \overline f \hat{v} \, dv $ as sources; we recall that by construction, they do not depend on $c$, but only on the lower bound $c_0$ (this fact is crucial). We observe that $(\overline f ,\overline E, \overline B)$ is a solution of the Vlasov-Maxwell system (with a suitable source in $\omega$). In addition, \eqref{Cond5.1} and \eqref{Cond5.2} clearly hold.

Let us define (recall \eqref{Th2C4})
\begin{equation*}
\mathfrak{b}=\frac{1}{c} \int_{\T^{2}} B_{0} \, dx = \int_{\T^{2}} b_{0} \, dx.
\end{equation*}
We introduce three different characteristics:
\begin{itemize}
\item the characteristics $(\overline{X},\overline{V})$ are associated with $\nabla \overline{\varphi} + \hat{v}^\perp \mathfrak{b}$,
\item the characteristics $(\tilde{X},\tilde{V})$ are associated with $\overline E - \tilde{E}+\frac{ \hat{v}^\perp}{c} (\overline{B} - \tilde{B})$,
\item the characteristics $(X,V)$ associated with $ \overline{E}+ \frac{\hat{v}^\perp}{c} \overline{B}$.
\end{itemize}
Let now show that if $c$ is large enough, the characteristics $(X,V)$ are close to the characteristics $(\overline{X}, \overline{V})$.
To that purpose we first prove that $(\tilde{X},\tilde{V})$ is close to $(\overline{X},\overline{V})$, and  that $(X,V)$ is close to $(\tilde{X},\tilde{V})$ (as $c \rightarrow + \infty$). \par
\ \par \noindent
{\bf 1.} Using the approximation Lemma \ref{approx} and defining $\tilde{B}$ and $\tilde{E}$ by \eqref{DefLesTildes}, we can choose $c$ large enough (larger than some $c_1(T)>0$), so that the characteristics $(\tilde{X},\tilde{V})$ are arbitrarily close to the characteristics $(\overline{X},\overline{V})$, and hence satisfy \eqref{Cond5.3}. 

Indeed, using \eqref{eqEstimB}-\eqref{eqEstimE}, we can set 
\begin{equation*}
\mathfrak{F}:= \overline{E} - \nabla \overline{\varphi} - \tilde{E} + \frac{ \hat{v}^\perp}{c} \left( \overline{B} - \tilde{B} - \int_{\T^{2}} B_{0} \, dx\right),
\end{equation*}
and observe that its norm can be made arbritrarily small, for $c$ large enough. Thus that the characteristics $(\tilde{X}, \tilde{V})$ satisfy \eqref{Cond5.3} follows from the definition of $(\overline{f}, \nabla \overline{\varphi})$ and an application of Lemma \ref{PropSolRefCM}. \par
\ \par \noindent
{\bf 2.} It remains to prove that, the characteristics $(X,V)$ are arbitrarily close to $(\tilde{X},\tilde{V})$ if $c$ is large enough (larger than some $c_2(T) \geq c_1(T)$). 
The convergences of $\tilde{E}$ and $\tilde{B}$ to zero as $c$ tends to infinity are only weak in time, but, as we will see, this is sufficient to establish this property on the characteristics.

Let us introduce
\begin{equation*}
(Y,W):=(X,V) - (\tilde{X},\tilde{V}).
\end{equation*}
Then $(Y,W)$ satisfies
\begin{equation*}
\left\{ \begin{array}{l}
\displaystyle
 \frac{dY}{dt} = W, \\
\displaystyle
 \frac{dW}{dt} = \overline{E}(X) - \overline{E}(\tilde{X})
 + \frac{\hat{V}^\perp}{c} \overline{B}(X) - \frac{\widehat{\tilde{V}}^\perp}{c}  \overline{B}(\tilde{X})
- \tilde{E}(\tilde{X}) - \frac{ \widehat{\tilde{V}}^\perp}{c} \tilde{B}(\tilde{X}).
\end{array} \right.
\end{equation*}
We remark that the right hand side of the equation of $W$ is bounded uniformly with respect to $c$. In the same way, we notice that $\tilde{X}$ is Lipschitz with respect to $t$, uniformly in $c$ and that $V \mapsto \hat{V}$ is Lipschitzian with a Lipschitz constant independent of $c$. Multiplying the equations by $Y$ and $W$ by $Y$ and $W$ respectively and integrating over time, one deduces that for some constant $C>0$ independent of $c$, one has
\begin{equation} \label{PreGronwall}
\frac{d}{dt} (Y^{2} + W^{2} ) \leq C (Y^{2} + W^{2} ) + \int_{0}^{t} \left( - \tilde{E}(\tilde{X}) - \frac{ \widehat{\tilde{V}}^\perp}{c} \tilde{B}(\tilde{X}) \right) \cdot W \, ds.
\end{equation}
Denote
\begin{equation*}
I_{1} := \int_{0}^{t} \tilde{E}(\tilde{X}) \cdot W \, ds \ \text{ and  } \ 
I_{2} := \int_{0}^{t}  \frac{ \widehat{\tilde{V}}^\perp}{c} \tilde{B}(\tilde{X}) \cdot W \, ds.
\end{equation*}
Let us show that uniformly on $[0,T]$, one has $I_{1} \rightarrow 0$ as $c \rightarrow +\infty$. Using the definition of $\tilde{E}$, and denoting $e_{k}(x)=\exp(i2\pi k \cdot x)$, we have
\begin{equation*}
I_{1} := \sum_{k \in \Z^{2} \setminus \{0 \}} \int_{0}^{t} W \cdot 
\left[
- \frac{k \wedge \hat{B}^{k}(0)}{|k|} \sin (sc|k|) 
+ \Big[\hat{E}^{k}(0) - \frac{ik}{|k|^{2}}\hat{\rho}^{k}(0) \Big] \cos (sc|k|)
\right] e_{k}(\tilde{X}) \, ds.
\end{equation*}
Integrating by parts, we infer
\begin{multline*}
I_{1} = \sum_{k \in \Z^{2} \setminus \{0 \}} \int_{0}^{t} \frac{1}{c|k|} \partial_{s}(W e_{k}(\tilde{X})) \cdot 
\left[
 \frac{k \wedge \hat{B}^{k}(0)}{|k|} \cos (sc|k|) 
- \Big[\hat{E}^{k}(0) - \frac{ik}{|k|^{2}}\hat{\rho}^{k}(0) \Big] \sin (sc|k|)
\right] \, ds \\
+ \sum_{k \in \Z^{2} \setminus \{0 \}} \frac{1}{c|k|}  W \cdot 
\left[
- \frac{k \wedge \hat{B}^{k}(0)}{|k|} \cos (tc|k|) 
+ \Big[\hat{E}^{k}(0) - \frac{ik}{|k|^{2}}\hat{\rho}^{k}(0) \Big] \sin (tc|k|)
\right] e_{k}(\tilde{X}) . \end{multline*}
Using $\partial_{s}(W e_{k}(\tilde{X})) = \partial_{s} W e_{k}(\tilde{X}) + W (\partial_{s} \tilde{X} \cdot k) i2 \pi e_{k}$ and the regularity of $B(0)$, $E(0)$ and $\rho(0)$, we obtain the claim that $I_{1} \rightarrow 0$ as $c \rightarrow +\infty$. The proof for $I_{2}$ is similar and therefore omited. \par
Using the convergences of $I_{1}$ and $I_{2}$, \eqref{PreGronwall} and Gronwall's lemma yields the claim on the characteristics. Therefore \eqref{Cond5.3} is satisfied for $c$ large enough. \par
\end{proof}
%
%
%
%
%
%
%
%
%
\subsection{With an arbitrary control set?}
The strip assumption is only used so that the local conservation of charge and the zero-mean current property are satisfied by the reference solution for Vlasov-Poisson, so that one can apply  the approximation Lemma \ref{approx}. More precisely, this geometrical hypothesis is only used to get \eqref{CirculationFluxNuls}, which in turn allows to ensure these properties.
For the case of an arbitrary open control set $\omega$, keeping the notations which follow Lemma \ref{PropAccelerePartout-mag}, we would need to construct a function $v$ satisfying the stronger property:
\begin{equation*}
\begin{aligned}
\int_\mathcal{D} v\cdot n \, dx= \int_\mathcal{D} v\cdot d\tau =0,\\
\int_{\mathcal{D}'} v\cdot n \, dx= \int_{\mathcal{D}'} v\cdot d\tau =0,
\end{aligned}
\end{equation*}
where $\mathcal{D}'$ is a line orthogonal to $\mathcal{D}$ and which does not meet $\omega$ (reduce $\omega$ if necessary).  Unfortunately, we were not able to build such a function. Maybe  the use of elliptic functions could be helpful in this general case.
\section{Construction of a solution of the non-linear problem}
\label{sec4}
To conclude the proof of Theorems \ref{TheoGCC}, \ref{Theo2} and \ref{Theo}, we can now rely on the reference solutions which were built in the previous sections to build a solution of the Vlasov-Maxwell system, which takes into account the initial data,  which remains ``close'' to the reference solutions and finally reaches the state $0$.

As the completions of the proofs are now very similar, we will focus mainly on the construction for Theorem \ref{Theo2} (and for Theorem \ref{Theo} which is a corollary of Theorem \ref{Theo2}). We will ultimately explain the needed modifications for Theorem \ref{TheoGCC}.

We will propose a functional framework which is bit different to those proposed for the case of the Vlasov-Poisson system with external magnetic field in \cite{GHK} and which is inspired by the work of Asano \cite{ASA} and Wollman \cite{Wol1,Wol2}. 


This section is organied as follows. We start by proving Theorem \ref{Theo2}. This is done in three steps. First, we construct a solution of the Vlasov-Maxwell system with source (together with the source itself) by a particular fixed point scheme (Subsection \ref{SubS:FPS}). Then we prove that this solution is relevant (Subsection \ref{Subsec:Rel}); this proves Theorem \ref{Theo2} when $f_{1}=0$ outside $\omega$. The general case $f_{1} \neq 0$ is explained in Subsection \ref{5.5}. Next in Subsection \ref{Subsec:TempsGrand} we deduce Theorem \ref{Theo}. In Subsection \ref{PrThGcc}, we prove Theorem \ref{TheoGCC}. The last subsections propose several possible generalizations of these results. \par

\subsection{Fixed point scheme} \label{SubS:FPS}
We recall that $x_0 \in \omega$ and $r_0>0$ are such that $B(x_0,2r_0) \subset \omega$. 
We introduce $R>0$ such that
\begin{equation} \label{Suppf0f1}
\mbox{Supp\,} f_{0}, \ \mbox{Supp\,} f_{1} \subset \T^{2} \times B(0,R/2).
\end{equation}
We apply Proposition \ref{refVM}; we fix $\overline{f}$ the reference function and $T$ the control time given in this proposition.
We will assume that $c \geq c(T)$. \par
Let $\varepsilon \in (0,1)$. We begin by introducing the operator ${\mathcal V}$ whose fixed point will give a solution to the nonlinear problem.
We first define the domain ${\mathcal S}_{\varepsilon,R}$ of
${\mathcal V}$ by
\begin{gather}
\nonumber
 \begin{array}{ll}
{{\mathcal S}_{\varepsilon,R}:= 
\Big\{ \ g \in C_b(Q_{T}) \ \Big/ \hfill} & {\hskip 8cm \hfill } \\  \end{array} \\
\label{DefS}
\begin{array}{ll}
{\mathbf a.\ } &
\forall t \in [0,T], \forall x \in \T^2, \ \operatorname{Supp}_v \, g(x,\cdot) \subset B(0,R), \\
{\mathbf b.\ } &
 {\|  g - \overline{f}
\|_{L^\infty([0,T]; H^3(\T^2 \times \R^2))} \leq \varepsilon, }\\
%
%
%
%
{\mathbf c.\ } &
 {\|  g - \overline{f}
\|_{W^{1,\infty}([0,T]; H^2(\T^2 \times \R^2))} \leq \varepsilon, }\\
{\mathbf d.\ } &
{\II \forall t \in [0,T],\ \int_{\T^{2} \times \R^{2}}  g(t,x,v)\,  dx \, dv = \int_{\T^{2} \times \R^{2}} f_{0}(x,v) 
dx \, dv\ ,}\\
{\mathbf e.\ } &
{ \forall t \in [0,T], \forall x \in \T^2, \partial_t \int_{\R^2} g \, dv + \operatorname{div}_x \int_{\R^2} \hat{v} \, g \, dv = 0,}\\
{\mathbf f.\ } &
{ g_{|t=0}=f_0 \Big\}.}
\end{array}
\end{gather}
%
%
Given $f_0$ small enough and satisfying \eqref{Suppf0f1}, one has $\tilde{f}+\overline{f} \in {\mathcal S}_{\varepsilon,R}$ (denoting by $\tilde{f}$ the solution of the relativistic free transport equation with initial data $f_0$), and consequently ${\mathcal
S}_{\varepsilon,R} \not = \emptyset$. From now on, this is systematically supposed to be the
case. \par
As in \cite{GLA,GHK}, we introduce the following subsets of $S(x_0,2 r_0)
\times \R^{2}$:
\begin{equation} \label{DefGammaMoins} 
\gamma^{-}:= \left\{ \II (x,v) \in S(x_0,2 r_0) \times
\R^{2} \ /\ |v| > \frac{1}{2} \text{ and } v\cdot \nu(x) <-\frac{1}{10}
|v|  \right\},
\end{equation}
\begin{equation} \label{DefGamma32Moins} 
\gamma^{2-}:= \left\{ \II (x,v) \in S(x_0,2 r_0) \times
\R^{2} \ /\ |v| \geq 1 \text{ and } v \cdot \nu(x) \leq -\frac{1}{8}
|v| \right\},
\end{equation}
\begin{equation} \label{DefGammaMoinsMoins} 
\gamma^{3-}:= \left\{ \II (x,v) \in S(x_0,2 r_0) \times
\R^{2} \ /\ |v| \geq 2 \text{ and } v \cdot \nu(x) \leq -\frac{1}{5}
|v| \right\},
\end{equation}
\begin{equation} \label{DefGammaPlus} 
\gamma^{+}:= \left\{ \II (x,v) \in S(x_0,2 r_0) \times
\R^{2} \ /\ v \cdot \nu(x) \geq 0 \right\},
\end{equation}
where $\nu(x)$ stands for the unit outward normal to the sphere $S(x_0,2 r_0)$ at point $x$.
One can readily check that
\begin{equation*}
\mbox{dist}( [S(x_0, 2 r_0) \times \R^2] \setminus \gamma^{2-}; \gamma^{3-}) >0.
\end{equation*}
\ \\
We introduce a $C^{\infty} \cap C_{b}^{1}$ smooth function $U:S(x_0,2 r_0)
\times \R^2 \to \R$, satisfying
\begin{equation} \label{DefU}
\left\{ \begin{array}{l}
{ 0 \leq U \leq 1, } \\
{U \equiv 1 \text{ in } [S(x_0, 2 r_0) \times \R^2] \setminus \gamma^{2-},} \\
{U \equiv 0 \text{ in } \gamma^{3-}.}
\end{array} \right.
\end{equation}%
We also introduce a cut-off function $\Upsilon:\R^{+} \to \R^{+}$, of
class $C^{\infty}$, such that
\begin{equation} \label{DefUpsilon}
\Upsilon = 0 \text{ in } \left[0,\frac{T}{48}\right] \cup  \left[\frac{47T}{48},T\right]  \ \text{ and } \
\Upsilon = 1 \text{ in } \left[\frac{T}{24},\frac{23T}{24}\right].
\end{equation} \par
Now, given $g \in {\mathcal S}_{\varepsilon,R}$, we associate $E^g,B^g$ in $C^0([0,T]; H^3(\T^2)) \cap C^1([0,T]; H^2(\T^2))$ as the solution of the Maxwell equations:
\begin{equation} \label{defMaxwell}
\left\{ \begin{aligned}
& \partial_t E^{g}_1 - c \partial_{x_2} B^{g}=-\int_{\R^2} g \hat v_1 \, dv, \quad \partial_t E^{g}_2 + c \partial_{x_1} B^{g}=-\int_{\R^2} g \hat v_2 \, dv,  \\
& \partial_t B^{g} +  \partial_{x_1} E^{g}_2 -  \partial_{x_2} E^{g}_1=0, \\
& \partial_{x_1} E^{g}_1 +  \partial_{x_2} E^{g}_2 =  \int_{\R^2} g \, dv- \int_{\R^2 \times \T^2} g \, dv \, dx,\\
& E^g|_{t=0} = E_0,  \quad B^g|_{t=0} = B_0.
\end{aligned} \right.
\end{equation} \par
Then, we define 
\begin{equation}
\label{nutilde}
\tilde{\mathcal{V}}(g):=f
\end{equation}
to be the solution of the following linear system
\begin{equation} \label{EqLin3}
\left\{ \begin{aligned}
& f(0,x,v) = f_{0} \text{ on } \T^{2} \times \R^{2}, \medskip \\
& \partial_{t} f + \hat{v}.\nabla_{x} f + (E^{g} + \frac{\hat{v}^\perp}{c} B^{g}).\nabla_{v} f = 0
\text{ in } [0,T] \times [(\T^{2} \times \R^{2}) \setminus \gamma^{-}], \medskip \\
& f(t^+,x,v) =[ 1- \Upsilon(t)] f(t^{-},x,v) + \Upsilon (t) U(x,v) f(t^{-},x,v) \text{ on } [0,T] \times \gamma^{-}.
\end{aligned} \right. 
\end{equation} 
\ \par
%
%
In order to explain the last equation of \eqref{EqLin3}, we introduce the characteristics $(X,V)$ associated with the force field $E^{g} + \frac{\hat{v}^\perp}{c} B^{g}$. 
In this equation, $f(t^{-},x,v)$ is the limit value of $f$ on
the characteristic $(X,V)(s,0,x,v)$ as time $s$ tends to
$t^-$. (Observe that for times smaller than $t$, but close to $t$, the corresponding
characteristic is not in $\gamma^{-}$.) 
When the characteristics $(X,V)$ meet $\gamma^-$ at time $t$, then
the value of $f$ at time $t ^+$ is fixed according to the last
equation in (\ref{EqLin3}). Loosely speaking, one can interpret the function $[ 1- \Upsilon(t)] + \Upsilon (t) U(x,v)$ as an ``opacity'' 
factor which varies according to time, the modulus of the velocity and to the angle of incidence of the trajectory on $S(x_{0}, 2r_{0})$. 
In this process, a part of $f$ is absorbed on $\gamma^{-}$,
which varies from the totality of $f$ to no absorption. Typically, when the velocity and the normal incidence of the trajectory are sufficient, the particles are absorbed.
We now consider a continuous linear extension operator
$$\overline{\pi}: \quad H^{3}([\T^{2}\setminus \overline{B}(x_{0},2r_{0})] \times \R^2;\R) \to
H^{3}(\T^{2}\times \R^2;\R),$$
which continuously maps $H^{s}([\T^{2}\setminus \overline{B}(x_{0},2r_{0})] \times \R^2;\R)$ to $H^{s}(\T^{2}\times \R^2;\R)$, for any $s \in [0,3]$ (by classical results, it is possible to consider such an extension operator because $\overline{B}(x_0,2 r_0)$ is smooth). In addition, it is clear that if a function $f$ has compact support in $v$, we can ensure that $\overline{\pi}(f)$ also satisfies this property.\par
Then we modify this operator in order to get a new operator $\tilde{\pi}$ which has the further property that for any integrable $f \in H^s((\T^{2}\setminus \overline{B}(x_{0},2r_{0})) \times \R^{2})$, one has
\begin{equation} \label{NeutralitedePi}
\int_{\T^{2} \times \R^{2}} \tilde{\pi}(f) \,  dv \, dx = \int_{\T^{2} \times \R^{2}}
f_{0}(x,v)\,  dv \, dx.
\end{equation}
This condition can easily be obtained by considering a regular, compactly supported,
nonnegative function $\mu$ with integral $1$ over $\overline{B}(x_{0},2r_{0}) \times \R^{2}$ and such that $\int \hat{v} \mu \, dv =0$, and adding to $\overline{\pi}(f)$ the function 
\begin{equation*}
\left[\int_{\T^{2} \times \R^{2}} f_{0} -\int_{\T^{2} \times \R^{2}} \overline{\pi} (f) \right] \mu.
\end{equation*}
We further modify the operator so that it satisfies the local conservation of charge. To this end, we add another regular distribution function $h$, compactly supported in $\overline{B}(x_0,2r_0) \times \R^2$, such that 
\begin{equation*}
\int h \, dv =0 
\ \text{ and } \
\div \int_{\R^{2}} h v \, dv = - \partial_t \int_{\R^{2}} \tilde{\pi}(f) \, dv - \operatorname{div} \int_{\R^{2}} \tilde{\pi}(f) v \, dv.
\end{equation*}
This can be done as in the proof of Proposition \ref{VPext}, but here this is much easier to build such a function since here we do not need the zero mean current condition \eqref{ZM}, and this can be achieved without making any assumption on the control set.
We finally obtain a continuous affine extension operator $\pi$. \par 
%
%
%
It is convenient to introduce another time-dependent function $\tilde{\Upsilon}$ depending on time such that:
\begin{equation} \label{DefUpsilon2}
\tilde{\Upsilon} = 0 \text{ in } \left[ 0, \frac{T}{100} \right]  \ \text{ and } \
\tilde{\Upsilon} = 1 \text{ in } \left[ \frac{T}{48}, T \right].
\end{equation}
Next, we introduce the continuous operator $\Pi$ 
given, for $f:([0,T] \times [ \T^{2} \setminus \overline{B}(x_{0},2r_{0})] \times \R^{2}) \cup ([0,T/48] \times \T^{2} \times \R^2) \rightarrow \R$, by:
\begin{equation} \label{RegleGrandPi}
(\Pi f)(t,x,v):= (1-\tilde{\Upsilon}(t)) f(t,x,v)
+ \tilde{\Upsilon}(t) [{\pi}f(t,\cdot,\cdot)](x,v).
\end{equation}
\par We finally define ${\mathcal V}[g]$ for $g \in {\mathcal S}_{\varepsilon,R}$ by:
\begin{equation} \label{DefV2}
{\mathcal V}[g]:= \overline{f} + \Pi ( f_{|([0,T] \times \left[
\T^{2} \setminus \overline{B}(x_{0},2r_{0}) \right] \times \R^{2}) \cup
([0,T/48] \times \T^{2} \times \R^2)} ) \text{ in } [0,T] \times \T^{2} \times \R^{2}.
\end{equation} \par \par

The goal is now to prove the existence of a fixed point for this operator. For clarity of exposure, we cut the argument into several lemmas.

\begin{lem}
\label{lem1}
The set $\mathcal{S}_{\epsilon,R}$, endowed with the $C^0_b$ topology, is a convex compact subset of $C^0_b(Q_T)$, and $\mathcal{V}$ is continuous in the $C^0_b$ topology.
\end{lem}
\begin{proof}[Proof of Lemma \ref{lem1}]
That $\mathcal{S}_\epsilon$ is a convex  subset of $C_b^{0}(Q_T)$ is straightforward. That it is compact can be proved exactly as in \cite[Section 3.3]{GLA} or \cite{GHK}, relying on Ascoli's theorem.

Likewise, the continuity can be proved as in \cite[Section 3.3]{GLA} or \cite{GHK} and we therefore omit the proof. Note that by Sobolev imbeddings, $E^g$ and $B^g$ are Lipschitz-continuous in $x$.
\end{proof}

\begin{lem} \label{lem2}
There exist $\epsilon_0>0$ and $R_0>0$ such that for any $0<\epsilon<\epsilon_0$ and $R>R_0$, we have:
\begin{equation*}
\mathcal{V}(\mathcal{S}_{\epsilon,R}) \subset \mathcal{S}_{\epsilon,R}.
\end{equation*}
\end{lem}
\begin{proof}[Proof of Lemma \ref{lem2}]
Let $g \in \mathcal{S}_{\epsilon,R}$: let us prove that $\mathcal{V}(g) \in  \mathcal{S}_{\epsilon,R}$.
That {\bf d.}, {\bf e.} and {\bf f.} hold is true by construction of $\mathcal{V}$.

That {\bf a.} and {\bf b.} hold comes from the fact that one can first prove:
\begin{equation*}
\begin{array}{ll}
{\mathbf a'.\ } &
\forall t \in [0,T], \forall x \in \T^2 \setminus \overline{B}(x_0,{2 r_0}), \ \operatorname{Supp}_v \, \tilde{\mathcal{V}}(g)(t,x, \cdot)  \subset B(0,R), \\
{\mathbf b'.\ } &
 {\|  \tilde{\mathcal{V}}(g) 
\|_{L^\infty([0,T]; H^3(\T^2 \setminus \overline{B}(x_0,{2 r_0}) \times \R^2))} \leq \varepsilon, }
\end{array}
\end{equation*}
as soon as $\epsilon$ is small enough and $R$ large enough. \par
These estimates are obtained with the same method as in the classical proof by Wollman \cite{Wol2} of the existence of solutions of the Vlasov-Maxwell system (without any source); special emphasis will be put on the new terms coming from the absorption procedure on ${S}(x_0,{2 r_0})$, that actually cause significant technical difficulties. The principle is to perform $L^2$ energy estimates separately on $[\T^2 \setminus \overline{B}(x_0,{2 r_0})] \times \R^2$ and on
$\overline{B}(x_0,{2 r_0}) \times \R^2$ (instead of on $\T^2 \times \R^2$ as in Wollman's paper), by using the Vlasov equation satisfied by 
$\tilde{\mathcal{V}}(g) $, and then make the sum of the two contributions.  The reason for this treatment is the discontinuity across the hypersurface $S(x_0,{2 r_0})$ which is caused by the absorption procedure.

In order to simplify the notations, we denote by $h_1$ the function $ \tilde{\mathcal{V}}(g) $ restricted to $[\T^2 \setminus \overline{B}(x_0,{2 r_0})] \times \R^2$ and by $h_2$ the function $\tilde{\mathcal{V}}(g) $ restricted to $B(x_0,{2 r_0}) \times \R^2$.

The estimates, first in $L^2$, then in $H^1$, etc., are obtained by induction. 
Let us start with the $L^2$ norm. By multiplying by $h_1$ the equation satisfied by $h_1$ and integrating over $[\T^2 \setminus \overline{B}(x_0,{2 r_0})] \times \R^2$, we obtain after an integration by parts:
\begin{align*}
\frac{d}{dt} \int_{[\T^2 \setminus \overline{B}(x_0,{2 r_0})] \times \R^2} |h_1|^2 \, dx \,dv =  \int_{S(x_0,{2 r_0})  \times\R^2} | h_1  |^2 \hat{v} \cdot \nu(x) \, d \sigma \, dv.
\end{align*}
Similarly, we obtain for $h_2$:
\begin{align*}
\frac{d}{dt} \int_{{B}(x_0,{2 r_0}) \times \R^2} |h_2|^2 \, dx \,dv = - \int_{S(x_0,{2 r_0})  \times\R^2} | h_2  |^2 \hat{v} \cdot \nu(x) d \sigma \, dv.
\end{align*}
For clarity, it is convenient to introduce the following notations:
\begin{align*}
\gamma^{in} :=  \{(x,v) \in S(x_0,{2 r_0}) \times \R^{2}, \ v \cdot (-\nu (x)) > 0\}, \\
\gamma^{out} := \{(x,v) \in S(x_0,{2 r_0}) \times \R^{2}, \ v \cdot (-\nu (x)) \leq 0\}.
\end{align*}

Using the definition of the absorption (see the last equation of \eqref{EqLin3}), we have
\begin{multline} \label{5.17}
\int_{\gamma^{in}} | h_1  |^2 v \cdot (-\nu(x)) \, d \sigma \, dv + \int_{\gamma^{out}} |  h_1   |^2 \hat{v} \cdot (-\nu(x)) \, d \sigma \, dv \\
- \int_{\gamma^{in}} |  h_2  |^2 v \cdot (-\nu(x)) \, d \sigma \, dv - \int_{\gamma^{out}} |  h_2   |^2\hat{v} \cdot (-\nu(x)) \, d \sigma \, dv \\
=\int_{\gamma^{in}} | h_1  |^2 \times [1-\alpha_1^2(t,x,v)] \hat{v}\cdot (-\nu(x)) \, d \sigma \, dv,
\end{multline}
for some smooth function $\alpha_{1}$ satisfying $0 \leq \alpha_1(t,x,v) \leq 1$. Therefore the right hand side of \eqref{5.17} is non-positive and one obtains $L^2$ estimates for  $\tilde{\mathcal{V}}(g)$:
\begin{align*}
\frac{d}{dt} \left(\int_{[\T^2 \setminus \overline{B}(x_0,{2 r_0})] \times \R^2} |h_1|^2 \, dx \,dv +  \int_{{B}(x_0,{2 r_0}) \times \R^2} |h_2|^2 \, dx \,dv\right) \leq 0.
\end{align*}
Thus, using \eqref{DefV2}, one obtains $L^2$ estimates for ${\mathcal{V}}(g)$:
$$
\int_{\T^2   \times \R^2} |{\mathcal{V}}(g)|^2 \, dx \,dv \leq C \| f_0 \|_{L^2(\T^2\times \R^2)}, \quad \forall t \in [0,T]. 
$$


Let us now treat the $H^1$ norm. We first differentiate the Vlasov equation with respect to space or velocity and we obtain as before (here $D:= \partial_{x_i}$ or $\partial_{v_i}$, $i=1,2$):
\begin{multline} \label{E1}
\frac{d}{dt} \int_{[\T^2 \setminus \overline{B}(x_0,{2 r_0})] \times \R^2} |D h_1|^2 \, dx \,dv 
=  \int_{S(x_0,{2 r_0})  \times\R^2} | D h_1  |^2 v \cdot \nu(x) \, d \sigma\, dv \\
- \int_{[\T^2 \setminus \overline{B}(x_0,{2 r_0})] \times \R^2} D\hat{v}\cdot \nabla_x h_1 Dh_1 \, dx \, dv
- \int_{[\T^2 \setminus \overline{B}(x_0,{2 r_0})] \times \R^2} D(E^g+ \frac{\hat v^\perp}{c}B^g)\cdot \nabla_v h_1 Dh_1 \, dx \, dv .
\end{multline}
Similarly, for $h_2$, we obtain the following estimate:
\begin{multline} \label{E2}
\frac{d}{dt} \int_{{B}(x_0,{2 r_0}) \times \R^2} |D h_2|^2 \, dx \,dv
=  - \int_{S(x_0,{2 r_0}) \times \R^{2}} | D h_2  |^2 v \cdot \nu(x) \, d \sigma \, dv \\
- \int_{{B}(x_0,{2 r_0}) \times \R^2} D\hat{v}\cdot \nabla_x h_2 Dh_2 \, dx \, dv
- \int_{{B}(x_0,{2 r_0}) \times \R^2} D(E^g+ \frac{\hat v^\perp}{c}B^g)\cdot\nabla_v h_2 Dh_2 \, dx \, dv .
\end{multline}
One can bound the last two terms of equations \eqref{E1} and \eqref{E2} (using the Cauchy-Schwarz inequality) by:
\begin{equation*}
C (1+ \| (E^g,B^g)\|_{L^\infty_t(H^3)})( \| h_1 \|_{H_1}^2+ \| h_2 \|_{H_1}^2).
\end{equation*}
It remains to estimate the first terms in the right hand sides of \eqref{E1} and \eqref{E2}.
One can use the definition of the absorption procedure in \eqref{EqLin3} to compute the derivatives  $\partial_t h_2$, $\partial_v h_2$ and $\partial_\tau h_2$ on $\gamma^{in}$ (denoting by $\tau$ a tangential unit vector on $S(x_0,{2 r_0})$) in terms of those of $h_1$ on $\gamma^{in}$. To obtain the ``missing'' derivative $\partial_{\nu} h_{2}$, one uses the Vlasov equation, for $(x,v) \in \gamma^{in}$. We can obtain the following relation on $\gamma^{in}$:
\begin{multline} \label{ChangementDerivee}
\nabla_{t,x,v} h_2(t,x,v) = 
\alpha(t,x,v) \nabla_{t,x,v}  h_1(t,x,v) + \beta(t,x,v,E^g,B^g) h_1(t,x,v), \\
\quad \text{for  }  t\in [0,T], x \in S(x_0,{2 r_0}), v \in \R^2,
\end{multline}
for some smooth bounded functions $\alpha$ and $\beta$ satisfying $0 \leq \alpha \leq 1$. We refer to \cite[Section 3.3]{GLA} for full details.

Summing the two contributions of $D h_1$ and $D h_2$, we deduce that the boundary terms can be rewritten using \eqref{ChangementDerivee}:
\begin{multline}
\int_{\gamma^{in}} |D  h_1  |^2 \hat{v} \cdot (-\nu(x)) \, d \sigma \, dv + \int_{\gamma^{out}} |D  h_1   |^2 \hat{v} \cdot (-\nu(x)) \, d \sigma \, dv \\
- \int_{\gamma^{in}} |D  h_2  |^2 \hat{v} \cdot (-\nu(x)) \, d \sigma \, dv - \int_{\gamma^{out}} |D  h_2   |^2 \hat{v} \cdot (-\nu(x)) \, d \sigma \, dv \\
=\int_{\gamma^{in}} | D h_1  |^2 \times [1-\alpha^2(t,x,v)] \hat{v} \cdot (-\nu(x)) \, d \sigma \, dv \\
- \int_{\gamma^{in}} \tilde\beta^2 |h_1|^2 \hat{v}\cdot (-\nu(x)) \,dv   -\int_{\gamma^{in}} \alpha \beta D h_1 h_1 \hat{v} \cdot (-\nu(x)) \, d \sigma \, dv. 
\end{multline}
The first term in the right hand side is non-negative. As we will see, the second one will be absorbed using a Gronwall estimate. We have to handle the last term, which is the most difficult one, in order to write it with the same form as the second one. We first claim that on $\gamma^{in}$, we can write (with again  $D:= \partial_{x_i}$ or $\partial_{v_i}$, $i=1,2$):
\begin{equation*}
D h_1 = a_1(t,x,v) \partial_t h_1 +  a_2(t,x,v) \partial_\tau h_1 +  a_3(t,x,v) \partial_v h_1,
\end{equation*}
where the $a_i$ are smooth functions (depending of course on the partial derivative $D$). This is simply proved by using the fact that $h_1$ is the solution of a Vlasov equation. Then we can compute the following integral:
\begin{multline*}
\int_0^t \int_{\gamma^{in}} \alpha \beta D h_1 h_1 \hat{v} \cdot (-\nu(x)) \, d \sigma \, dv \, ds   
=\int_0^t \int_{\gamma^{in}} \alpha \beta a_1(t,x,v) \partial_t \frac{h_1^2}{2} \hat{v} \cdot (-\nu(x)) \, d \sigma \, dv \, ds \\
+ \int_0^t \int_{\gamma^{in}} \alpha \beta a_2 (t,x,v) \partial_\tau  \frac{h_1^2}{2} \hat{v} \cdot (-\nu(x)) \, d \sigma \, dv \, ds \\
+ \int_0^t \int_{\gamma^{in}} \alpha \beta a_3 (t,x,v) \partial_v \frac{h_1^2}{2} \hat{v} \cdot (-\nu(x)) \, d \sigma \, dv \, ds.
\end{multline*}
At this point, we can integrate by parts with respect to $t$, $\tau$ and $v$, which yields only two boundary terms in time (at time $0$ and $t$ respectively, coming from the integration by parts in time) and terms of the form:
$$
\int_0^t \int_{\gamma^{in}} \beta_2(s,x,v) |h_1|^2 \hat{v} \cdot (-\nu(x)) \, d \sigma \,dv \, ds,
$$
for some smooth bounded function $\beta_2$.
The boundary terms in time are the two following one:
\begin{equation*}
\int_{\gamma^{in}} \alpha_{|s=0} \beta_{|s=0} f_0^2 \hat{v}\cdot (-\nu(x))\, d \sigma \, dv,
\end{equation*}
and
\begin{equation*}
\int_{\gamma^{in}} \alpha_{|s=t} \beta_{|s=t}  h_{1|s=t}^2 \hat{v}\cdot (-\nu(x))\, d \sigma \, dv.
\end{equation*}
To treat the second term, we use a trace estimate and Young's inequality, as follows:
\begin{align*}
\int_{\gamma^{in}} \alpha \beta  h_1^2 \hat{v} \cdot (-\nu(x)) \, dx\, dv & \leq C \| h_1 \|_{L^2}  \| h_1 \|_{H^1} \\
& \leq  C \| h_1 \|^2_{L^2} + \frac{1}{100}  \| \nabla h_1 \|^2_{H^1}. 
\end{align*}

Using this computation we obtain:
\begin{multline*}
\frac{1}{2} \sum_{D \in \{\partial_{x_i},\partial_{v_i}\}} \left(\int_{[\T^2 \setminus \overline{B}(x_0,{2 r_0})] \times \R^2} |D h_1(t)|^2 \, dx \,dv +  \int_{{B}(x_0,{2 r_0}) \times \R^2} |D h_2(t)|^2 \, dx \,dv \right) \\
\leq C \| f_{0} \|_{L^{2}} + C  \sum_{D \in \{\partial_{x_i},\partial_{v_i}\}}\int_0^t  \int_{[\T^2 \setminus \overline{B}(x_0,{2 r_0})] \times \R^2} |D h_1|^2 \, dx \,dv \, ds 
+  C \| h_1 \|^2_{L^2}  \\
+ \sum_{D \in \{\partial_{x_i},\partial_{v_i}\}} \left(  \int_{[\T^2 \setminus \overline{B}(x_0,{2 r_0})] \times \R^2} |Df_0|^2 \, dx \,dv 
+ \int_{{B}(x_0,{2 r_0}) \times \R^2} |D f_0|^2 \, dx \,dv\right),
\end{multline*}
and we conclude by Gronwall's lemma. \par
\ \par
The higher order derivatives are obtained in the same way by induction.
In particular, we observe that we can obtain a similar formula as \eqref{ChangementDerivee} for higher derivatives. Indeed, let us prove that for $i,j \in \{t,x_1,x_2,v_2,v_2\}$ we have:
\begin{equation} \label{ChangementDerivee2}
\partial^2_{i,j} h_2 = \alpha \partial_{i,j} h_1 + \beta^1_{i,j} \partial_ j h_1  + \beta^2_{i,j} \partial_ i h_1 + \gamma_{i,j} h_1 
\quad \text{on  } [0,T] \times S(x_0,2r_0) \times \R^2,
\end{equation}
where $\alpha$ is the same function as in \eqref{ChangementDerivee} and $\beta^1_{i,j},\beta^2_{i,j} ,\gamma_{i,j}$ are smooth bounded functions of $(t,x,v)$. From \eqref{ChangementDerivee}, we obtain easily that \eqref{ChangementDerivee2} is true for any $j$ when $i$ is a tangent direction to $[0,T] \times S(x_0,2r_0) \times \R^2$. For the remaining (normal) direction, we differentiate the Vlasov equation satisfied by $h_1$ and $h_2$:
\begin{equation*}
\partial_t (\partial_j h_k) + \hat{v} \cdot \nabla_x (\partial_j h_k) + F^g \cdot \nabla_v (\partial_j h_k) 
= -\left[ \partial_j \hat{v} \cdot \nabla_x h_k + \partial_j F^g \cdot \nabla_v h_k\right],
\end{equation*}
denoting by $F^g := E^g + \frac{\hat{v}^\perp}{c} B^g$ and $k=1,2$. Then we replace the first order derivatives of $h_2$ using \eqref{ChangementDerivee} and the second order  derivatives using \eqref{ChangementDerivee2} except the one involving the normal direction in $\hat{v} \cdot \nabla_x (\partial_j h_2)$, that is $(\hat{v}\cdot \nu)\nu \cdot \nabla_x (\partial_j h_2)$. As a result, we can express this last term in terms of $\alpha \, (\hat{v}\cdot \nu)\nu \cdot \nabla_x (\partial_j h_1)$ and lower order terms as in \eqref{ChangementDerivee2}.

Then, when one studies the contributions from $\gamma^{-}$ in the $H^2$ estimates, the most harmful terms come as before from the crossed terms $\alpha \partial_{i,j} h_1  \beta_{i,j} \partial_j h_1$, but as before, one can use an integration by parts argument. 
Similar statements hold for derivatives of order $3$. Note that when proving the $H^{2}$ and the $H^{3}$ estimates, the time derivatives have to be understood in the weak sense.

Once {\bf a'.} and {\bf b'.} are obtained, that  {\bf a.}, {\bf b.} hold for ${\mathcal{V}}(g)$ is a consequence of the construction of the extension operator.

There remains to prove that  {\bf c.} also holds. This is actually a consequence of {\bf b'.} and the fact that $h_1$ solves the Vlasov equation, so that 
$$\|\partial_t  h_1 \|_{L^\infty([0,T]; H^2([\T^2 \setminus \overline{B}(x_0,{2 r_0})] \times \R^2))} \leq C \varepsilon.$$
Finally that {\bf c.} holds is obtained by construction of the extension operator.

\end{proof}
We conclude this section with the following statement.
\begin{lem}
\label{lem3}
If $R>R_0$ and $0<\epsilon<\epsilon_0$, $\mathcal{V}$ has a fixed point in $\mathcal{S}_{\epsilon,R}$.
\end{lem}
\begin{proof}[Proof of Lemma \ref{lem3}]
This is an immediate consequence of Lemma \ref{lem1}, Lemma \ref{lem2} and the Schauder fixed point theorem.
\end{proof}
\subsection{Relevancy of the solution} \label{Subsec:Rel}
Let us now prove that the solution we have just built is relevant for the controllability problem of Theorem \ref{Theo2}.
\begin{lem} \label{lem4}
There exists $\epsilon_1>0$ such that for any $0<\epsilon<\epsilon_1$, all the characteristics $(X,V)$ associated with $\frac{\hat{v}^\perp}{c} B^f + E^{f}$, where $f$ is a fixed point of ${\mathcal V}$ in ${\mathcal S}_{\varepsilon,R}$, meet $\gamma^{3-}$ for some time in $[\frac{T}{12},\frac{11T}{12}]$.
\end{lem}
\begin{proof}[Proof of Lemma \ref{lem4}]
This is a consequence of Proposition \ref{refVM} and the fact that a fixed point $(f,E^f,B^f)$ is a small perturbation of the reference solution $(\overline{f},\overline{E},\overline{B})$. When $\epsilon$ is chosen small enough, we have:
\begin{equation*}
\begin{aligned}
\|B^f- \overline{B} \|_\infty \ll 1, \\
\| E^f - \nabla{\overline{\phi}} \|_\infty \ll 1.
\end{aligned}
\end{equation*}
This involves that the characteristics $(X,V)$ associated with the force field $E^f + \frac{\hat{v}^\perp}{c} B^f$ still satisfy the property:
\begin{equation} \label{returnouf}
\forall x \in \T^2, \forall v \in \R^2, \exists t \in [T/11,10T/11], \quad X(t,x,v) \in B(x_0,3r_0/2), \quad \text{and  } |V|(t,x,v) \geq 3.
\end{equation}
As in \cite{GHK}, we can prove that this implies that, provided that $\varepsilon@$ is small enough, the trajectories meet $\gamma^{3-}$ during $[\frac{T}{12},\frac{11T}{12}]$, so that all particles are absorbed.
The technical points are almost identical to those of the proof given in  \cite[Lemma 5.3]{GHK}, and so we omit them.
\end{proof}
We can finally end the proof of Theorem \ref{Theo2} when $f_{1}=0$ outside $\omega$.
\begin{proof}[Proof of Theorem \ref{Theo2} when $f_{1}=0$ outside $\omega$]
By Lemma \ref{lem3} and Lemma \ref{lem4}, we obtain a solution $f$ such that:
\begin{equation*}
\operatorname{Supp}[f(T,\cdot,\cdot)] \subset \omega \times \R^2.
\end{equation*}
We obtain furthermore that the distribution function ${f}$ satisfies for some function $G$ supported in $\omega$:
\begin{equation*}
\partial_t {f} + \hat{v} \cdot \nabla_x {f} +  \operatorname{div}_v \left[(E^{{f}} + \frac{1}{c}\hat{v}^\perp  B^{{f}} )f\right]
= G \text{ in } [0,T] \times \T^{2} \times \R^{2}.
\end{equation*}
The function $G$ is then our control function. Steering finally the state exactly to $f_{1}$ can be performed via a function supported inside $\omega$, which also yields a source term internal to $\omega$.
This ends the proof of Theorem \ref{Theo2} (in the case where the target $f_1 = 0$ outside $\omega$).
\end{proof}
\subsection{From $f_1 = 0$ to general targets}
\label{5.5}
Up to now, the proof of Theorem \ref{Theo2} has been completed for the case $f_1 = 0$ outside $\omega$. We cannot directly invoke a reversibility argument to get the general case, since we cannot precisely say at which state the electro-magnetic field $(E,B)$ is at final time.

Nevertheless, it is possible to modify the fixed point scheme to complete the proof in the general case.
Keeping the same notations, 
we introduce a function $\hat{f}$ which is the solution of the Vlasov equation:
\begin{equation} \label{EqLin4}
\left\{ \begin{array}{l}
\hat{f}(0,x,v) = f_{1} \text{ on } \T^{2} \times \R^{2}, \medskip \\
\partial_{t} \hat{f} + \hat{v}\cdot\nabla_{x} \hat{f} + (E^{g} + \frac{\hat{v}^\perp}{c} B^{g})\cdot\nabla_{v} \hat{f}= 0
\text{ in } [0,T] \times [(\T^{2} \times \R^{2}) \setminus \gamma^{-}], \medskip \\
\hat{f}(t^+,x,v) =[ 1- \Upsilon(t)] \hat{f}(t^{-},x,v) + \Upsilon (t) U(x,v) \hat{f}(t^{-},x,v) \text{ on } [0,T] \times \gamma^{-}.
\end{array} \right. 
\end{equation}
We observe, as in the proof of Theorem \ref{Theo2} (using the property of the characteristics), that $\hat{f}_{|t=T} \equiv 0$.
 

We finally define the fixed point operator ${\mathcal V}$ which maps $g \in {\mathcal S}_{\varepsilon,R}$ to:
\begin{equation} \label{DefV3}
{\mathcal V}[g]:= \overline{f} + \Pi ( (f+\hat{f}(T-t,x,-v)) _{|([0,T] \times \left[ \T^{2} \setminus B(x_{0},2r_{0}) \right] 
\times \R^{2}) \cup ([0,T/48] \times \T^{2} \times \R^2)} ) \text{ in } [0,T] \times \T^{2} \times \R^{2},
\end{equation}
and conclude as previously.
\subsection{Proof of Theorem \ref{Theo}} 
\label{Subsec:TempsGrand}
As already explained, Theorem \ref{Theo} is a consequence of Theorem \ref{Theo2}.
We recall that the speed of light here is equal to $1$. Let us show how to reduce to a sufficiently large $c$. The principle is to rely on the scale invariance of the equations. Indeed, taking $\lambda$ large enough in \eqref{lamb} increases the value of the ``effective'' speed of light.

Let us be more specific. Given $\omega$ (satisfying the strip assumption) and $b_{0}$, we apply Theorem \ref{Theo2}. We obtain a time $T_{0}$, and apply the result to $T=T_{0}$. We obtain a constant $\kappa > 0$ and a speed of light $c_{T_0}>0$, such that whenever $c \geq c_{T_0}$, starting from $f_{0}, E_{0}, f_{1}$ satisfying \eqref{Th2C1}-\eqref{Th2C4}, we have a solution of Vlasov-Maxwell with speed of light $c$, steering $f_{0}$ to $f_{1}$ in time $T_{0}$. 

We define $\lambda_{1} := 1/ c_{T_0}$. For $\lambda \in (0, \lambda_{1}]$, we observe that
$$c := 1/\lambda$$ 
satisfies $c \geq c_{T_0}$. 
We define $\kappa_1:= \min(1, \lambda^{2}) \kappa$. Given $\tilde{f}_{0}, \tilde{f}_{1}, \tilde{E}_{0}$ satisfying \eqref{Th3C1}-\eqref{Th3C2} and such that
\begin{equation*}
\| (\tilde{f}_0,\tilde{E}_0)\|_{H^3} \leq \kappa', \quad \| \tilde{f}_1\|_{H^3} \leq \kappa', 
\end{equation*}
then 
$$f_0(x,v):=\tilde{f}_0(x, \lambda v), \quad f_1(x,v):=\tilde{f}_1(x, \lambda v), \quad E_{0}(x) := 1/\lambda^{2} \tilde{E}_{0}(x)$$ 
also satisfy \eqref{Th2C1}-\eqref{Th2C2} and moreover \eqref{Th2C3}. We also introduce the initial magnetic field $\tilde{B}_0$  (and its rescaled counterpart $B_0$):
$$\tilde{B}_0(x) := \lambda b_0 (x), \quad B_0(x) := 1/\lambda^2 \tilde{B}_0(x).$$
Therefore, we observe that:
$$B_0(x) = c b_0(x)$$
and we can apply Theorem \ref{Theo2} to obtain a control function $G_{\lambda}$ such that the solution $f^{\lambda}$ of the system:

\begin{equation*}
\left\{ \begin{aligned}
& \partial_t f^{\lambda} + \hat{v}_\lambda\cdot\nabla_x f^{\lambda} +  \operatorname{div}_v \left[(E^{\lambda}+ {\lambda}\hat{v}_\lambda^\perp B^{\lambda}) f^{\lambda}\right] = \mathbbm{1}_\omega G_\lambda, \\
& \partial_t E^{\lambda} + \frac{1}{\lambda} \operatorname{curl} B^{\lambda} = -\int_{\R^2} f^{\lambda}\hat v_\lambda \, dv,  \quad \partial_t B^{\lambda} + \frac{1}{\lambda} \operatorname{curl} E^{\lambda}=0, \\
& \operatorname{div} E^{\lambda}  = \int_{\R^2} f^{\lambda} \, dv- \int_{\R^2\times \T^2} f^{\lambda} \, dv \, dx, 
 \quad \operatorname{div} B^{\lambda} = 0, \\
& f^{\lambda}_{| t=0} = f_0, \quad E^{\lambda}_{| t=0} =  E_0, \quad B^{\lambda}_{| t=0} = B_0 \, \,(=c b_0), \\
\end{aligned} \right.
\end{equation*}
with $\hat{v}_{\lambda} = \frac{v}{\sqrt{1+ \frac{|v|^2}{(1/\lambda)^2}}}$, satisfies:
\begin{equation}
f^{\lambda}_{| t=T_0} =f_1.
\end{equation}
One deduces a solution $(f, E, B)$ to the Vlasov-Maxwell in the original variables by performing the ``inverse'' change of variables: 
$$
f( t/\lambda,x, \lambda v):=f^\lambda(t,x,v), \quad
E( t/\lambda,x)  :=\lambda^{2} E^{\lambda}(t,x), \quad
B( t/\lambda,x)  := \lambda^{2} B^{\lambda}(t,x).
$$
In particular, we observe that
\begin{equation*}
f_{| t= T_0/\lambda} =\tilde{f}_{1}.
\end{equation*}
which gives the controllability in large time of Theorem \ref{Theo}. 
\subsection{Proof of Theorem \ref{TheoGCC}} 
\label{PrThGcc}

In order to prove Theorem \ref{TheoGCC}, we modify the previous construction as follows.
First, we recall that in the construction of the reference solution in Section~\ref{sec2}, we have considered the sequence of sets $\overline{\omega'} \subset \omega'' \subset \overline{\omega''} \subset \omega$. Let $\overline{f}$ be the reference solution given by Proposition \ref{propGCC}. \par
Let us describe the principle of the absorption in this case. We consider a covering of balls $(B(x,r_0))_{x \in \omega''}$ of $\overline{\omega''}$, whose radius  $r_0$  is small enough so that the covering with balls of radius $2r_0$ is included in $\omega$. By compactness, we can assume that this covering is made of a finite number of balls $B(x_1, r_0), ... , B(x_k,r_0)$. \par
On each ball $B(x_i, 2 r_0)$ (for $i = 1,...,k$) we impose the same absorption procedure as before. Then one can proceed with almost the same proof as Theorem \ref{Theo2}. The only difference appears in the proof of Lemma \ref{lem2}: we now have to estimate the norms $L^2$, $H^1$, etc, in \emph{each} connected component of $\T^2 \setminus \cup_{i=1}^k S(x_k,2r_0)$. Remark that there is only a \emph{finite} number of such regions. Then, we make the sum of all contributions of each region to get the relevant estimates. \par
Finally, we obtain the construction of a control supported in $\omega$ which brings the state $f$ of the system to a value satisfying:
\begin{equation*}
\operatorname{Supp}[f(T,\cdot,\cdot)] \subset \omega \times \R^2,
\end{equation*}
for $T$ large enough.
Taking  $\tilde{T} > T$ large enough so that one can apply Theorem \ref{controlMaxwell} (control of linear Maxwell equations), we can further modify this solution so that
\begin{equation*}
\operatorname{Supp}[f(t,\cdot,\cdot)] \subset \omega \times \R^2, \quad \text{for any  } t \in [T,\tilde{T}],
\end{equation*}
and so that at time $\tilde{T}$, we have
$$
f(\tilde{T}, \cdot, \cdot) \equiv 0 \quad \text{on  } (\T^2\setminus \omega) \times \R^2,
$$
and  
$$(E(\tilde{T},\cdot), B(\tilde{T},\cdot))=(0,0).$$ 
By a slight abuse of notation, we still write $T$ instead of $\tilde{T}$. We clearly obtain that the distribution function ${f}$ satisfies for some function $G$ supported in $\omega$:
\begin{equation*}
\partial_t {f} + \hat{v}.\nabla_x {f} +  \operatorname{div}_v \left[ (E^{{f}} + \frac{1}{c}\hat{v}^\perp B^{{f}} )f \right]
= G \text{ in } [0,T] \times [\T^{2}  \setminus \omega ] \times \R^2.
\end{equation*}
The function $G$ is then our control function, which brings the state to $(0,0,0)$. \par
We can finally directly invoke a reversibility argument (take $\lambda= -1$ in \eqref{rescale}) or proceed as in Subsection \ref{5.5}, to steer the system to general $(f_1,E_1,B_1)$.


%
%
%
%
%
%
\subsection{A $C^1$ or $H^s$ Cauchy theory ?}
\label{gene0}

In the Vlasov-Poisson case, in order to build the solution in a $C^1$ framework, we used the same absorption procedure which roughly consists in absorbing fast  particles which enter the control zone with sufficient normal incidence.
This procedure is a bit costly in terms of regularity, and it causes a loss of moments in velocity. In \cite{GLA, GHK}, this is compensated via interpolation by a loss of regularity. Nevertheless, this regularity is re-gained afterwards using the elliptic regularity of the Poisson equation. One first difficulty in the Maxwell case lies in the fact that this is not possible to do so (as the Maxwell equations are hyperbolic). A more severe obstruction lies in the hypothetical application of the Glassey-Strauss-Schaeffer $C^1$ Cauchy theory for Vlasov-Maxwell. Since the Maxwell equation can be seen as a wave equation, it is quite remarkable to be able to build solutions in a $C^1$ framework (even worse, in the wave equation form, there seems to be  a loss of derivative). Actually, one can observe that the theorems of Glassey-Strauss-Schaeffer crucially rely on the fact that the distribution function \emph{is a solution} to a Vlasov equation (without source). This allows to trade derivatives like $(\partial_t + \hat{v} \cdot \nabla_x)$ for derivatives in $v$. Then, integrations by parts in $v$ allow these authors to perform suitable $C^1$ estimates. For our problem, this is unfortunately impossible to do so, since in the control zone, the Vlasov will feature a source (the control function) with only $C^{0}$ regularity. Therefore it seems difficult to build solutions with these methods.

Instead of the $C^1$ theory, we have thus opted for $H^3$ theory of Asano and Wollman (although it is less sharp in terms of regularity). 

\subsection{Removing the smallness assumption on $B$ in Theorem \ref{TheoGCC}}
\label{gene1}

Let us explain how one can remove the smallness assumption on $B_0$ (and $B_1$) in Theorem \ref{TheoGCC}. When one applies the return method, instead of building a reference solution which goes from $(0,0,0)$ to $(0,0,0)$, one builds instead a solution going from $(0,0,B_i)$ to $(0,0,0)$. To do so, one can still rely on the controllability of the Maxwell equations. Then we apply the same procedure as above.
\subsection{Generalization to data with unbounded support in velocity in Theorem \ref{TheoGCC}}
\label{gene2}
Until now, we have dealt with distribution functions having a bounded support in velocity. As claimed before,  we expect the theorems contained in this paper to be generalized to distribution functions with unbounded support in velocity, but having exponential decay at infinity. 

To that purpose, instead of the standard Sobolev space $H^3$ of the local theory of Wollman \cite{Wol2}, we may use weighted Sobolev spaces (accounting for the exponential decay in velocity), following the work of Asano \cite{ASA}. It is likely (computations get more tedious) that the fixed point procedure proposed in this section is still relevant.

%
%
%
%
%
%
%
%
%
%
\section{Appendix: Proof of Lemma \ref{PropSolRefCM}}
%
%

In this section, we provide a proof of Lemma  \ref{PropSolRefCM}.  Actually we establish a result which is more general, by allowing more generic magnetic fields. Let us first describe the class of magnetic fields $\mathfrak{b}$ that are admissible.
\bigskip
\begin{defi}
\label{defben}
If a magnetic field $\mathfrak{b}$ is such that either $\mathfrak{b}$ or $-\mathfrak{b}$ satisfies the conditions {\bf 1.} and {\bf 2.} below, we say that $\mathfrak{b}$ satisfies the {\bf bending condition}. 
\end{defi}
\noindent
{\bf Condition  1. Geometric control condition on $\mathfrak{b}$.} We assume that there exists a compact set $K$ of $\mathbb{T}^2$ on which $b>0$ and which satisfies the geometric control condition:
\begin{equation} \label{GeometricCondition}
\forall x \in \mathbb{T}^2, \ \forall e \in \mathbb{S}^{1}, \ \exists y \in \R^{+} \text{ such that } x+ye \in K.
\end{equation}
For a compact subset $K$ of $\T^{2}$ and $r>0$ we denote
\begin{equation} \label{Kr}
K_{r} := \{ x \in \T^{2} \ / \ d(x,K) \leq r \}.
\end{equation}
The first geometric assumption can be reinterpreted with the help of the following lemma (see \cite{GHK} for an elementary proof).
\begin{lem} \label{LemEpaissi}
Let $K \subset \T^{2}$ such that $\mathfrak{b} >0$ on $K$ and satisfying \eqref{GeometricCondition}.
Then there exists $\underline{b}>0$, $d>0$ and $D>0$ such that
\begin{equation} \label{Petitd}
\mathfrak{b} \geq \underline{b} \ \text{ on } K_{2d},
\end{equation}
\begin{equation} \label{GrandD}
\forall x \in \T^{2}, \ \forall e \in \S^{1}, \ \exists t \in [0,D],
\ \forall s \in \left[t,t+ \frac{d}{2}\right], \  x + se \in K_{d}.
\end{equation}
\end{lem}
Roughly speaking, this lemma gives, for a ray of light (whose velocity is a unit vector), the maximum time it can spend outside $K_d$ during one passage (that corresponds to $D$) and the minimal time it has to spend in $K_d$ during one passage (that corresponds to $d/2$). \par
This allows us to introduce the second condition on the magnetic field: \par
\noindent
{\bf Condition  2. Bound from below.} We assume that there exists $\tilde{b} \in \R^-$ such that
\begin{equation*}
 \mathfrak{b} \geq \tilde{b}\ \text{ on } \T^2,
\end{equation*}
and such that, keeping the same notations as in Lemma \ref{LemEpaissi},
\begin{equation} \label{1.14}
 \tilde{b} D +  \underline{b} \frac{d}{2}  \ \, > 0.
\end{equation}
It is clear that a constant nonzero magnetic field satisfies the bending condition. \par
A less general notion was introduced in \cite{GHK} (the second condition was replaced by a fixed sign assumption, which corresponds to $\tilde{b}=0$); the main interest of this new definition comes from the fact that it is stable by small perturbations. Indeed, if some $\mathfrak{b}$ satisfies the bending condition, then there exists $\epsilon_0>0$ such that any $b$ with $\|b - \mathfrak{b} \|_\infty< \epsilon_0$, still satisfies the bending condition. In particular, the following lemma allows to improve some of the results of \cite{GHK}. \par
\begin{lem} \label{PropSolRefCM2}
Let $x_0 \in \T^2, r_0 >0$ and  $\overline{M}>1$. Given $\mathfrak{b} \in C^{1}(\T^{2})$ satisfying the bending condition, there exist:
\begin{itemize}
\item $c_{0}>0$ depending on $\mathfrak{b}, x_0, r_0$, 
\item $\underline{m} >0$ depending only on $\mathfrak{b}, x_0, r_0$, 
\item$T>0$ depending on $\mathfrak{b}$, $x_0, r_0$ and $\overline{M}$, and 
\item $\kappa$ depending on $\mathfrak{b}$, $x_0, r_0$ and $\overline{M}$,
\end{itemize}
such that for all $\mathfrak{F} \in L^{\infty}(0,T;W^{1,\infty}(\mathbb{T}^{2} \times \mathbb{R}^{2}))$ satisfying $\| \mathfrak{F} \|_{L^{\infty}} \leq \kappa$, if $c \geq c_{0}$ then the characteristics $(\overline{X}, \overline{V})$ associated with ${\hat{v}^\perp} \mathfrak{b} +\mathfrak{F}$ satisfy:
\begin{multline} \label{GVCM2}
\forall x \in \mathbb{T}^2, \forall v \in \mathbb{R} ^2 \text{ such that }  \overline{M} \geq | v | \geq \underline{m},
\exists t \in (T/4,3T/4), \ \overline{X}(t,0,x,v) \in B(x_0,r_0/2) \\
 \text{ and for all } s \in [0,T], \ \ \frac{|v|}{2} \leq |\overline{V}(s,0,x,v)| \leq 2|v|.
\end{multline}
\end{lem}
\ \par
We prove Lemma \ref{PropSolRefCM2} in several cases of increasing complexity. In a first time (Cases 1--3), we suppose that $\mathfrak{F}=0$. In Case 4, we explain how to take the additional small force $\mathfrak{F}$ into account. \par
In all cases, we define
\begin{equation} \label{DefOvelineEta}
\overline{b}:= \max_{x \in \T^{2}} \mathfrak{b}(x).
\end{equation}
Before giving the proof, we give its main ideas. \par
\ \par
\noindent {\it Idea of the proof of Lemma \ref{PropSolRefCM2}. }The proof is a variant of  that given for \cite[Proposition 5.1]{GHK}; here we have to be careful of the effects of relativistic transport and in addition, the assumptions on $\mathfrak{b}$ are more general. Conversely, the assumption of smallness on $\mathfrak{F}$ will allow us to simplify some aspects of the proof.

In a first time, we consider non-relativistic trajectories ($c=+\infty$) and $\mathfrak{F}=0$.
One can show that for (non-relativistic)  free transport (whose trajectories are  straight lines), there is only a finite number of bad initial directions which produce trajectories which never meet the ball $B(x_0,r_0)$. Indeed, trajectories with directions having irrational slopes are dense in the torus $\T^2$ and therefore have to meet $B(x_0,r_0)$. Considering periodic trajectories with rational slopes, because of the ``thickness'' of the ball $B(x_0,r_0)$, only a finite number of them produce trajectories which do not meet the ball. A magnetic field satisfying the bending condition allows to circumvent these bad initial directions. The proof of Lemma \ref{PropSolRefCM2} is based on the three following ingredients.
\begin{itemize}
\item Let us define the classical characteristics $({X}, {V})$ associated with ${{v}^\perp} \mathfrak{b}$.
We can first observe that when the modulus of the initial velocity $|v|$ is large, the characteristics $({X}, {V})$ remain close to those of free transport (straight lines) at least for a time of order $1/|v|$. This allows to show that trajectories $X$ which initially have a velocity with a good direction, meet the control zone as it is the case for free transport.

\item For trajectories which initially have a bad direction, we rely on the ``rotation'' effect provided by the magnetic field, which yields that bad directions become good directions after some time $\tau$. The condition {\bf 1.} ensures that trajectories ``often'' meet the zone where $\mathfrak{b}>0$, in which the direction of the velocity gets rotated and therefore do not remain of bad direction. Note that at each passage in the zone $\mathfrak{b}>0$, the angle of the velocity is only modified by $\mathcal{O}(1/|v|)$, but using the reinterpretation of Lemma \ref{LemEpaissi}, we can see that during a time of order $1$, there are at least $\mathcal{O}(|v|)$ passages in this zone. On the other hand, the condition {\bf 2.} in the bending condition allows to say the trajectories are not too affected by their passages in the zones where $\mathfrak{b}$ can be non-positive (and which could otherwise annihilate any bending effect). 

\item Finally, by choosing $c$ sufficiently large and $\mathfrak{F}$ sufficiently small, by a perturbation argument, we get that the relativistic trajectories $(\overline{X}, \overline{V})$ are close to the classical trajectories $(X,V)$, which allows to conclude.

\end{itemize}

\begin{figure}[!ht]
\begin{center}
	\resizebox{!}{4cm}{\input{TrajO3.pstex_t}}
\end{center}
\begin{caption}{An illustration of the proof of Lemma \ref{PropSolRefCM2}: the bending effect of the magnetic field.} \end{caption}
\end{figure}

\begin{proof}[Proof of Lemma \ref{PropSolRefCM2}]
\ \par
%
In the sequel, we denote $(X^{\#},V^{\#})$ the characteristics associated with the relativistic free transport and $(\overline{X},\overline{V})$ the characteristics associated with relativistic transport with the force ${\hat{v}^\perp} \mathfrak{b} +\mathfrak{F}$.
\ \par
\noindent{\bf Case 1.} \emph{Constant magnetic field and $\mathfrak{F}=0$.}
Let us first suppose that $\mathfrak{b}$ is constant and $\mathfrak{F}=0$; for readability we assume here that $\mathfrak{b}:=1$.

As noticed in  \cite[Appendix A, p. 373-374]{GLA}, there is only a finite number of directions in $\mathbb{S}^1$ for which there exists a half-line in $\mathbb{T}^2$ which does not intersect $B(x_0,r_0/8)$. Identifying $\mathbb{S}^1$ with $[0,2\pi[$, we denote them $\alpha_1,...,\alpha_N \in [0,2\pi[$.
%
\ \par 
We introduce the neighborhoods of $\alpha_{i}$:
\begin{equation*}
\mathcal{V}_i=(\alpha_i- \beta/2, \alpha_i + \beta/2),
\end{equation*}
as follows. Let $\beta>0$ and $\tau >0$  satisfying:
\begin{equation} \label{defbt}
\beta < \min_{i\neq j} d(\mathcal{V}_i,\mathcal{V}_j)/8
\ \text{ and } \ 
{\tau} = \sqrt{{1+ \frac{M^2}{c_0^2}}}\min_{i\neq j} d(\mathcal{V}_i,\mathcal{V}_j)/7.
\end{equation}
By construction, observe that: 
$$ \beta < \frac{\tau}{\sqrt{{1+ \frac{M^2}{c_0^2}}}}.$$
By a compactness argument, there exists $L>0$ such that for any $x \in \mathbb{T}^2$, and for any $a_i \in \mathbb{S}^1\setminus \cup_{i=1}^N \mathcal{V}_i$, any  trajectory $X^{\#}$ starting from $x$ with a direction $a_i$ has to cover at most a distance $L$ to meet $B(x_0,r_0/8)$.

We choose $c_0$ and $m$ large enough such that:
\begin{equation} \label{cond1}
T_{m}:=\frac{L \sqrt{1 + m^2/c^2}}{m}\leq L \sqrt{\frac{1}{m^2} + \frac{1}{c_0^2}}<\tau.
\end{equation}
This is the time ``free'' trajectories $X^{\#}$ with velocity of modulus $m$ take to cover the distance $L$. We observe that for any $| v | \geq m$, we have 
\begin{equation} \label{Eq:Tv}
T_{| v |}:=\frac{L \sqrt{1 + |v|^2/c^2}}{|v|}\  \leq T_m.
\end{equation}
Now let $x\in \mathbb{T}^2, v \in \mathbb{R}^2$ with $|v| \geq {m}$. First note that for all $t$, $|\overline{V}(t,0,x,v)|=|v|$. Now let us discuss according to the direction of $v$. \par
\ \par \noindent
$\bullet$ First case: $\frac{v}{| v |} \in  \mathbb{S}^1\setminus \cup_{i=1}^N \mathcal{V}_i$.

We have, for any $t \leq T_{|v|} \leq T_m$, 
\begin{align*}
| X^{\#}(t,0,x,v)- \overline{X}(t,0,x,v)| & \leq \frac{| v |}{\sqrt{1 + |v|^2/c^2}} \frac{T_{| v |}^2}{2}= \frac{1}{\sqrt{1 + |v|^2/c^2}} \frac{L^2 (1 + |v|^2/c^2)}{2|v|}\\
 & \leq  \frac{L^2}{2}\sqrt{1/c_0^2+ 1/m^2} .
\end{align*}
We can impose $c_0$ and $m$ large enough such that:
\begin{equation} \label{cond2}
\frac{L^2}{2}\sqrt{1/c_0^2+ 1/m^2} < r_0/8.
\end{equation}
As a result, we obtain that
$$\exists t \in (0,T_{m}], \,  \overline{X}(t,0,x,v) \in B(x_0,r_0/4),$$
and \eqref{GVCM2} is trivial here since $|\overline{V}(t,0,x,v)|$ is conserved. \par
\ \par \noindent
$\bullet$ Second case:  $\frac{v}{| v |} \in  \cup_{i=1}^N \mathcal{V}_i$, say $\mathcal{V}_j$. \par
As already explained in the heuristics, we wait for a time $\tau$, that is, we consider:
\begin{equation*}
(x',v'):=(\overline{X} (\tau,0,x,v),\overline{V} (\tau,0,x,v)).
\end{equation*}
We observe that because of the ``rotation'' induced by the magnetic field, the angle between $v$ and $v'$ is equal to $\tau /\sqrt{1+ |v|^2/c^2}$, which is larger than $\tau/( \sqrt{1+ M^2/c_0^2})$. \par
\ \par
Consequently, due to the choice of $\beta$, 
\begin{equation*}
\frac{v'}{| v' |} \in  \mathbb{S}^1\setminus \cup_{i=1}^N \mathcal{V}_i,
\end{equation*}
and thus we are in the same case as before. \par
\ \par
Therefore, defining $T:= 4(T_m + \tau)$, we have proven that (actually up to a harmless translation in time):
\begin{equation*}
\exists t \in (T/4,3T/4], \ \overline{X}(t,0,x,v) \in B(x_0,r_0/4),
\end{equation*}
together with the condition \eqref{GVCM2} on the velocity field. \par
\ \par
\noindent{\bf Case 2.}  \emph{Positive magnetic field modulus and $\mathfrak{F}=0$.} Here we suppose that $\fb = \fb(x)>0$ on $\T^{2}$.

We are in the case where in Lemma \ref{LemEpaissi}, we can take $K=K_{d}=\T^{2}$ and 
\begin{equation*}
\underline{b} =\inf_{x \in \mathbb{T}^2} \fb>0.
\end{equation*}
Keeping the same notations as before, we set $\tau \in (0,T]$ and $\beta>0$ in order that
\begin{equation} \label{defibt}
\beta < \min_{i\neq j} d(\mathcal{V}_i,\mathcal{V}_j)/8
\ \text{ and } \ 
{\tau} = \frac{\sqrt{{1+ \frac{M^2}{c_0^2}}}}{\underline{b}}\min_{i\neq j} d(\mathcal{V}_i,\mathcal{V}_j)/7.
\end{equation}
 The proof is very similar to the previous one. Indeed, the following comparison estimate still holds for $t \leq T_{|v|} \leq \tau$ where $T_{|v|}$ was defined in \eqref{Eq:Tv}:
\begin{equation} \label{56b}
| X^{\#}(t,0,x,v)- \overline{X}(t,0,x,v)| \leq  \overline{b} \frac{L^2}{2}\sqrt{1/c_0^2+ 1/m^2}.
\end{equation}
Let $\tilde{x}\in \mathbb{T}^2, \tilde{v} \in \mathbb{R}^2$. We distinguish as before between two possibilities.
Using the inequality \eqref{56b}, the first case holds identically for $m$ large.
For the second case, we just have to check that with this magnetic field, the velocity is rotated by an angle at least equal to $\beta$ after some time $t  \leq \tau$. \par
We use the following computation for general $(x,v)$. Denote by $\theta(t)$ the angle between $v^{\perp}$ and $\overline{V}(t,0,x,v)$. We compute the scalar product of $\overline{V}(t,0,x,v)$ and  $\frac{d \overline{V}(t,0,x,v)}{dt}$, using the identity:
\begin{equation*}
\frac{d \overline{V}(t,0,x,v)}{dt}= \fb(\overline{X}(t,0,x,v)) \frac{\overline{V}(t,0,x,v)^\perp}{\sqrt{1+ |\overline{V}(t,0,x,v)|^2/c^2}}.
\end{equation*}
We straightforwardly obtain that $| \overline{V}(t,0,x,v)|=| v |$. Then, taking the scalar product of $v^{\perp}$  and  $\frac{d \overline{V}(t,0,x,v)}{dt}$, we likewise obtain:
%
\begin{equation} \label{Thetaprime}
\theta'(t) =  \frac{\fb(\overline{X}(t,0,x,v))}{\sqrt{1+ |v|^2/c^2}},
\end{equation}
(even if $\sin \theta(t)=0$ in which case one considers the scalar product with $v$.)
We deduce that $\theta'(t) \geq \underline{b}/{\sqrt{1+M^2/c_0^2}}$. \par
Thus going back to $(\tilde{x},\tilde{v})$, by the intermediate value theorem and the definition of the neighborhoods ${\mathcal V}_{i}$ and to \eqref{defibt}, there is a positive time $t$ less or equal to $\tau$ for which we have:
 \[
 \overline{V}(t,0,\tilde{x},\tilde{v}) \in   \mathbb{S}^1\setminus \cup_{i=1}^N \mathcal{V}_i,
 \]
and we conclude as in Case {\bf 1.}

\ \\
\par
\noindent{\bf Case 3.}  \emph{Magnetic field satisfying the bending condition and $\mathfrak{F}=0$.} Let us now consider the general case for $\bf$, but still without the additional force $\mathfrak{F}$. \par
Given $K$ satisfying the geometric condition \eqref{GeometricCondition}, we introduce $d$ and $D$ as in Lemma \ref{LemEpaissi}. 
We consider $K_{d}$ as in \eqref{Kr}. 
Using the condition {\bf 2.} (bound from below), we have: 
\begin{equation} \label{gamma}
\gamma:= D \tilde{b}+ \frac{d}{2}  \underline{b} \ \, > 0.
\end{equation}
(we recall that by definition, $\tilde{b}\leq 0$).
We assume here that in addition to \eqref{defbt},  $\tau \in (0,T]$ and $\beta>0$ are such that
\begin{equation*}
\beta < \frac{\tau\gamma}{2\sqrt{1+ M^2/c_0^2}(D+d/2)}. 
\end{equation*}
%
%
%
\ \par
Let $x\in \mathbb{T}^2$ and $v \in \mathbb{R}^2$. Once again we distinguish between the two possibilities (good or bad direction).
As before the first case is still similar since \eqref{56b} is still valid. We have to give a new argument only for the second case. \par
\ \par
%
%
To that purpose, we examine the behavior of the characteristics during the time interval $[0,\tau]$. \par
%
%
\ \par
\noindent
$\bullet$ By \eqref{GrandD}, each passage in $\T^2\setminus K$ of $X^{\#}(t,0,x,v)$ lasts at most during a time equal to:
\begin{equation*}
D\sqrt{1+ |v|^2/c^2}/ | v |.
\end{equation*}
The characteristics $\overline{X}$ are not straight lines since they are modified by the magnetic field. Let us prove nevertheless that if $| v |$ and $c$ are large enough, then the particle $\overline{X}(t,0,x,v)$ can remain at most during a time $D\sqrt{1+ |v|^2/c^2}/| v |$ in $\T^{2} \setminus K_{d}$. 
Let $x \in \T^{2} \setminus K_{d}$, and $\frac{v}{| v |} \in \mathbb{S}^1$, let $\sigma>0$. By Lemma \ref{LemEpaissi}, there exists some $s< D\sqrt{1+ |v|^2/c^2}/| v | $ such that $X^{\#}(\sigma + s,\sigma,x,v) \in K$. Now we can evaluate as before:
\begin{equation*}
\left| X^{\#}(\sigma+ s ,\sigma,x,v) - \overline{X}(\sigma + s ,\sigma,x,v) \right| \leq \overline{b}  \frac{D^2}{2}\sqrt{1/c_0^2+ 1/|v|^2} .
\end{equation*}
Hence, we can impose $c_0$ and $m$ large enough such that when $| v | \geq m$, 
\begin{equation*}
\overline{X}(\sigma + s ,\sigma,x,v) \in K_{d}.
\end{equation*}
Hence, we deduce that each passage of $\overline{X}(t,0,x,v)$ in $\mathbb{T}^2 \setminus K_{d}$ lasts at most during a time equal to $D\sqrt{1+ |v|^2/c^2}/| v |$, which proves the claim. \par
%
%

\ \par

\noindent $\bullet$ Likewise, given $x \in \T^{2} \setminus K_{d}$,  $\frac{v}{| v |} \in \mathbb{S}^1$ and $\sigma>0$, by Lemma \ref{LemEpaissi}, there exists some $s< D\sqrt{1+ |v|^2/c^2}/| v | $ such that $X^{\#}(\tau,\sigma,x,v) \in K$ for all $\tau \in [s+\sigma,s+\sigma+\frac{d}{2} \sqrt{1+|v|^{2}/c^{^2}}/|v|]$.
%
%
We have for all $\tau \in [s+\sigma,s+\sigma+\frac{d}{2} \sqrt{1+|v|^{2}/c^{^2}}]/|v|]$ that 
\begin{equation*}
| X^{\#}(\tau,\sigma,x,v)- \overline{X}(\tau,\sigma,x,v)| \leq \overline{b} \frac{| v |}{\sqrt{1+ |v|^2/c^2}} \frac{\left( (D+\frac{d}{2})  \sqrt{1+ |v|^2/c^2} / |v|) \right)^2}{2}.
\end{equation*}
%
Hence we can choose $m$ and $c_{0}$ large enough such that for any $| v | \geq m$ and when $c \geq c_{0}$, $\overline{X}(\tau,\sigma,x,v) \in K_{d}$ for all $\tau \in [s+\sigma,s+\sigma+\frac{d}{2} \sqrt{1+|v|^{2}/c^{^2}}]$.

\ \par

\noindent $\bullet$ We see that a trajectory alternates between passages in $K_d$ and in $\T^{2} \setminus K_{d}$. We denote as previously by $\theta(t)$ the angle between $\overline{V}(t,0,x,v)$ and $v$ and aim at finding a lower bound for $\theta(\tau)$. We introduce
\begin{equation*}
N:= \left\lfloor \frac{\tau |v| }{\sqrt{1+ |v|^2/c^2}(D+d/2)} \right\rfloor,
\end{equation*}
(here $\lfloor \cdot \rfloor$ denotes the floor function) and
\begin{equation*}
A := \Big\{ t \in [0,\tau] \ \Big/ \ \overline{X}(t,0,x,v) \in \T^{2} \setminus K_{d }\Big\}.
\end{equation*}
There are at most $N$ distinct intervals of length $\sqrt{1+ |v|^2/c^2}(D+d/2)$ in $[0,\tau]$.
Using the considerations above, one sees that the Lebesgue measure of $A$ satisfies
\begin{equation*}
|A| \leq (N+1) D\sqrt{1+ |v|^2/c^2}/| v |.
\end{equation*}
Indeed, taking $t_{0} := \inf A$, one has 
\begin{equation*}
|A \cap [t_{0}, t_{0} + \sqrt{1+ |v|^2/c^2}(D+d/2)] | \leq D\sqrt{1+ |v|^2/c^2}/| v |,
\end{equation*}
and we can make the same reasoning with $t_{1} := \inf (A \cap \{ t \geq t_{0} + \sqrt{1+ |v|^2/c^2}(D+d/2) \}$, etc. Likewise, 
\begin{equation*}
|([0,\tau] \setminus A) \cap [t_{0}, t_{0} + \sqrt{1+ |v|^2/c^2}(D+d/2)] | \geq \frac{d}{2}\sqrt{1+ |v|^2/c^2}/| v |,
\end{equation*}
and
\begin{equation*}
|[0,\tau] \setminus A | \geq N \frac{d}{2}\sqrt{1+ |v|^2/c^2}/| v |,
\end{equation*}

For $|v|\geq m$, with $m$ large enough and $c_{0}$ large enough we have $N \geq 3$ and hence 
\begin{equation*}
N -1 \geq \frac{\tau |v| }{2 \sqrt{1+ |v|^2/c^2}(D+d/2)},
\end{equation*}
so
\begin{equation*} 
\begin{aligned}
\theta(\tau) & \geq \underline{b} \, |[0,\tau] \setminus A | + \tilde{b} \, |A| \\
& \geq N \frac{d}{2} \underline{b} \sqrt{1+ |v|^2/c^2} / |v| + (N+1) D \tilde{b} \sqrt{1+ |v|^2/c^2} / |v| \\
& \geq  N \Big(D\tilde{b}+ \frac{d}{2} \underline{b} \Big) \sqrt{1+ |v|^2/c^2} / |v| + \frac{D \tilde{b}}{| v |}\sqrt{1+ |v|^2/c^2}  \\
& \geq  (N-1) \Big(D\tilde{b}+ \frac{d}{2} \underline{b} \Big) \sqrt{1+ |v|^2/c^2} / |v| \\
& \geq  \frac{\tau}{2\sqrt{1+ |v|^2/c^2}(D+d/2)} \Big(D \tilde{b}+ \frac{d}{2} \underline{b} \Big) \\
&= \frac{\tau\gamma}{\sqrt{1+ M^2/c_0^2}(D+d/2)} > \beta \, \, (>0). 
\end{aligned}
\end{equation*}
Finally we can conclude as in Case {\bf 2.}


%
%
%

%
%
%
%
%
\ \\
\par
\noindent{\bf Case 4.}  {\it With a nontrivial additional force $\mathfrak{F}$.} 

Let us finally explain how one can take the small force $\mathfrak{F}$ into account. 
First, we study the equations for $|\overline{V}|$ and $\theta$, where $\theta$ is the angle between $v$ and $\overline{V}(t,0,x,v)$. The following computations are valid for $v$ large so that $|\overline{V}(t,0,x,v)|$ does not vanish and for a time interval where $\theta \in [-\pi/2,\pi/2]$. \par
\ \par \noindent
$\bullet$ For what concerns $|\overline{V}|$, it suffices to take the scalar product with $\overline{V}(t,0,x,v)$ of the equation of $\overline{V}$. We infer
\begin{equation*}
\frac{d}{dt} |\overline{V}(t,0,x,v)|^{2} = 2 \mathfrak{F} \cdot \overline{V}(t,0,x,v),
\end{equation*}
so that
\begin{equation} \label{EvolNorme}
\frac{d}{dt} |\overline{V}(t,0,x,v)| =  \frac{\mathfrak{F} \cdot \overline{V}(t,0,x,v)}{| \overline{V}(t,0,x,v) |}.
\end{equation}
Therefore, it is clear that for $T>0$, there exists $\kappa>0$ small enough, such that if the $L^\infty$ norm of $\mathfrak{F}$ is smaller than $\kappa$, then one has for all $(x,v) \in \mathbb{T}^2 \times \mathbb{R}^2$ \text{ with } $| v | \geq m$, for any $t\in [0,T]$,
\begin{equation} \label{CompTailev}
\frac{|v|}{2} \leq |\overline{V}(t,0,x,v)| \leq 2 |v|.
\end{equation}
\ \par \noindent
$\bullet$ For what concerns $\theta$, taking the scalar product of the equation of $\overline{V}$ with $v$ we deduce
\begin{multline*}
\left(\frac{d}{dt} |\overline{V}(t,0,x,v)| \right) |v| \cos \theta(t) - |\overline{V}(t,0,x,v)||v|\theta'(t) \sin \theta(t) \\
= \mathfrak{b}(\overline{X}(t,0,x,v)) \frac{\overline{V}^\perp(t,0,x,v) \cdot v}{\sqrt{1+ \frac{|\overline{V}(t,0,x,v)|^2}{c^2}}} + \mathfrak{F} \cdot v.
\end{multline*}
Hence
\begin{multline*}
|\overline{V}(t,0,x,v)||v|\theta'(t) \sin \theta(t) \\ 
= \mathfrak{b}(\overline{X}(t,0,x,v)) |\overline{V}(t,0,x,v)| |v| \sin (\theta(t))  \frac{1}{\sqrt{1+ \frac{|\overline{V}(t,0,x,v)|^2}{c^2}}}
 - \mathfrak{F} \cdot \left(v - \frac{\overline{V}(t,0,x,v) |v|}{|\overline{V}(t,0,x,v)|}\cos \theta(t)\right).
\end{multline*}
We notice that
\begin{equation*}
v - \frac{\overline{V}(t,0,x,v) |v|}{|\overline{V}(t,0,x,v)|}\cos \theta(t) 
= \mbox{p}_{\{ \overline{V}(t,0,x,v)\}^{\perp}} (v),
\end{equation*}
where $\mbox{p}_{\{ \overline{V}(t,0,x,v) \}^{\perp}} (v)$ denotes the orthogonal projection of $v$ on $\{ \overline{V}(t,0,x,v) \}^{\perp}$.
So
\begin{equation} \label{EvolAngle}
\theta'(t)  = \frac{\mathfrak{b}(\overline{X}(t,0,x,v))}{\sqrt{1+ \frac{|\overline{V}(t,0,x,v)|^2}{c^2}}},  + \frac{1}{|\overline{V}(t,0,x,v)|} \mathfrak{F} \cdot \frac{\mbox{p}_{\{ \overline{V}(t,0,x,v)\}^{\perp}} (v) }{|v|\sin \theta(t)}.
\end{equation}
Note that 
\begin{equation*}
| \mbox{p}_{\{ \overline{V}(t,0,x,v)\}^{\perp}} (v)| = | v | \, |\sin(\theta(t))|,
\end{equation*}
so that, using Cauchy-Schwarz inequality:
\begin{equation*}
-\frac{1}{|\overline{V}(t,0,x,v)|} \left|\mathfrak{F} \cdot \frac{\mbox{p}_{\{ \overline{V}(t,0,x,v)\}^{\perp}} (v) }{|v|\sin \theta(t)}\right| \geq -\frac{1}{|\overline{V}(t,0,x,v)|} \Vert \mathfrak{F} \Vert_{\infty}.
\end{equation*}
\ \par
Now it is rather straightforward to revisit Case {\bf 3.} to include the small force $\mathfrak{F}$.
A helpful ingredient is given by the following Gronwall estimate, in which we compare the characteristics $(\overline{X},\overline{V})$ associated with $\mathfrak{F}+\mathfrak{b}(x)v^{\perp}$ with the characteristics $({X},{V})$ associated with the magnetic field $\mathfrak{b}(x)v^{\perp}$ alone. We have:
\begin{equation} \label{Gronw}
\left\{ \begin{array}{l}
	|V(\sigma+ t,\sigma,x,v) -\overline{V}(\sigma+ t,\sigma,x,v)| \leq \|\mathfrak{F} \|_{\infty} \exp(\| \mathfrak{b}\|_{W^{1,\infty}} (1+2|v|) t), \\
	|X(\sigma+ t,\sigma,x,v) -\overline{X}(\sigma+ t,\sigma,x,v)| \leq t \|\mathfrak{F} \|_{\infty} \exp(\| \mathfrak{b}\|_{W^{1,\infty}} (1+2|v|) t),
\end{array} \right.
\end{equation}
which allows to get the result by using a perturbation argument (we recall that $\|\mathfrak{F} \|_{\infty} \leq \kappa)$.
\end{proof}
\bibliographystyle{amsplain}
\bibliography{cvm}
\end{document}